\documentclass[french,11pt, twoside]{article}
\usepackage[latin1]{inputenc}
\usepackage[T1]{fontenc}
\usepackage{lmodern}
\usepackage[a4paper]{geometry}
\usepackage{mathtools}
\usepackage{amssymb}
\usepackage{shuffle}
\usepackage{amsthm}
\usepackage{mathrsfs}
\usepackage{latexsym}
\usepackage{stmaryrd}
\usepackage{booktabs}
\usepackage{ifthen}
\usepackage{pgfmath}
\usepackage{babel}

\newcommand{\N}{\mathbb{N}} % entiers naturels
\newcommand{\K}{\mathbb{K}} % corps de base
\newcommand{\B}{\mathcal{B}} % définir une base d'un ev 
\DeclareMathOperator{\Prm}{Prim}
\newcommand{\Prim}[2]{\ifthenelse{\equal{#2}{}}{\Prm(#1)}{{\Prm(#1)}_{#2}}} % éléments primitifs de #1 de degré #2
\DeclareMathOperator{\Gn}{Gén}
\newcommand{\Gen}[2]{\ifthenelse{\equal{#2}{}}{\Gn(#1)}{{\Gn(#1)}_{#2}}} % éléments générateurs de #1 de degré #2
\newcommand{\Idl}[1]{\mathcal{#1}} % Noter un idéal de Hopf d'une algèbre de Hopf
\newcommand{\FQSym}{\mathbf{FQSym}} % fonctions quasi-symétriques libres
\newcommand{\QSym}{\mathbf{QSym}} % fonctions quasi-symétriques
\newcommand{\NSym}{\mathbf{NSym}} % dual de QSym
\newcommand{\Sn}[1]{\mathfrak{S}_{#1}} % groupe symétrique de {1,...,#1}
\newcommand{\tens}[1]{T\langle #1 \rangle} % écrire une algèbre tensorielle
\newcommand{\dual}[1]{#1^{\circledast}} % écriture du dual gradué de #1
\newcommand{\ot}{\otimes} % raccourci pour écrire le produit tensoriel
\newcommand{\De}{\Delta} %  raccourci pour écrire le delta du coproduit
\newcommand{\Rq}[1]{\ifthenelse{\equal{#1}{}}{\subparagraph{Remarques.}}{\subparagraph{Remarque.}}} % remarque(s)
\newcommand{\Ex}[1]{\ifthenelse{\equal{#1}{}}{\subparagraph{Exemples.}}{\subparagraph{Exemple.}}} % exemple(s)
\newcommand{\Cex}{\subparagraph{Contre-exemples.}} % contre-exemple(s)
\newcommand{\ExCex}{\subparagraph{Exemples et contre-exemples.}} % Exemples et contre-exemple(s)
\newcommand{\CQMM}{théorème de Cartier-Quillen-Milnor-Moore} % pour ne pas à avoir à écrire les nom à la main à chaque fois.
\newcommand{\PH}{Poincaré-Hilbert} % pour ne pas à avoir à écrire les nom à la main à chaque fois
\newcommand{\Lieb}{\mathcal{L}eib}
\newcommand{\Zinb}{\mathcal{Z}inb}
\newcommand{\Com}{\mathcal{C}om}
\newcommand{\PreLie}{\mathcal{P}ré\mathcal{L}ie}
\newcommand{\Quad}{\mathcal{Q}uad}

\input{xy}% permet de dessiner notamment des diagrammes commutatifs.
\xyoption{all}

\DeclareMathOperator{\AAlph}{Alph}
\DeclareMathOperator{\IAlph}{IAlph}
\DeclareMathOperator{\Irr}{Irr}
\DeclareMathOperator{\pack}{pack}

\newtheorem{defi}{\indent Définition}

\newtheorem{cor}[defi]{\indent Corollaire}
\newtheorem{theo}[defi]{\indent Théorème}
\newtheorem{prop}[defi]{\indent Proposition}

\newcommand{\PX}{\pack(X^{*})}
\newcommand{\WMat}{\mathbf{WMat}}
\newcommand{\SH}{\mathfrak{S}\mathcal{H}}
\newcommand{\SPW}{\mathbf{SPW}}
\newcommand{\ISPW}{\mathbf{ISPW}}
\newcommand{\Pn}[1]{\mathcal{P}r_{#1}} %notation pour un sev de Prim(\ISPW_{n})
\newcommand{\Cp}{\mathcal{C}_{e} }% notation pour l'ev des compsitions étendues.

\begin{document}
\title{A propos de l'algèbre de Hopf des mots tassés $\WMat$}
\date{}
\author{Cécile Mammez\\
\small{\it Univ. Littoral Côte d'Opale,}\\
\small{\it EA 2597 - LMPA Laboratoire de Mathématiques Pures et Appliquées Joseph Liouville,} \\ 
\small{\it F-62228 Calais, France and CNRS, FR 2956, France}\\
\small{e-mail : mammez@lmpa.univ-littoral.fr}
}

\maketitle

\textbf{Résumé.} Dans cet article, on étudie l'algèbre de Hopf des mots tassés $\WMat$ introduite par Duchamp, Hoang-Nghia et Tanasa. Pour cela on commence par considérer $\WMat$ à proprement parler (absence de coliberté et description du dual). Puis, on s'intéresse à une sous-algèbre de Hopf de permutations, notée $\SH$, dont le dual $\dual{\SH}$ est muni d'une structure de quadri-algèbre et donc d'une double structure d'algèbre dendriforme. On introduit par la suite $\ISPW$, une algèbre de Hopf de mots tassés stricts croissants. Son caractère cocommutatif pousse à s'intéresser à ses éléments primitifs. On en décrit quelques familles. On montre ensuite que $\ISPW$ et l'algèbre de Hopf $\NSym$ des fonctions symétriques non commutatives sont isomorphes. On définit par la suite une algèbre de Hopf quotient de compositions étendues $\Cp$. Celle-ci n'est pas cocommutative mais ses éléments primitifs sont liés à ceux de $\ISPW$. De plus, on exprime $\Cp$ comme un coproduit semi-direct d'algèbres de Hopf. Cette construction met à jour une coaction permettant par la suite de définir deux actions de groupes. On termine par la construction d'un isomorphisme explicite entre $\dual{\ISPW}$ et $\QSym$ conduisant à un isomorphisme explicite entre $\ISPW$ et $\NSym$. \\

\textbf{Mots-clés.} Algèbres de Hopf combinatoires, mots tassés, permutations, mots tassés stricts croissants, compositions, coaction, coproduit semi-direct, action, morphismes, fonctions quasi-symétriques. \\

\textbf{Abstract.} In this article we study the packed words Hopf algebra $\WMat$ introduced by Duchamp, Hoang-Nghia et Tanasa. We start by explaining that $\WMat$ is not cofree, giving its antipode and describing its graded dual. We consider then a Hopf sub-algebra of permutations called $\SH$. Its graded dual $\dual{\SH}$ has a quadri-algebra structure, so it has a double dendriform algebra structure too. Thereafter, we introduce  $\ISPW$, a Hopf algebra of increasing strict packed words. It is graded, connected and cocommutative so is isomorphic to the enveloping algebra of its primitive elements. We describe some families of primitive elements. We prove that $\ISPW$ and non commutative symmetric functions are isomorphic.  We define then an extended compositions Hopf algebra  $\Cp$. It is not cocommutative but its primitive elements and those from $\ISPW$ are linked.  We give an interpretation of $\Cp$ in terms of a semi-direct coproduct Hopf algebra. By using this, we can define two actions groups. We finish by giving an explicit isomorphism between $\dual{\ISPW}$ and $\QSym$ and another one between  $\ISPW$ and  $\NSym$. \\

\textbf{Keywords.} Combinatorial Hopf algebras, packed words, permutations, increasing strict packed words, compositions, coaction, semi-direct coproduct, action, morphisms, quasi-symmetric functions. \\

\textbf{AMS classification.} 16T30.\\

\tableofcontents
\allowdisplaybreaks
\section*{Introduction}
Le procédé d'extraction-contraction est largement utilisé dans la construction de cogèbres: on peut par exemple penser à la théorie des nombres avec l'algèbre de Hopf des diagrammes de dissection introduite par Dupont dans \cite[chapitre 2]{Dupont2014}, à la théorie des champs quantiques avec l'algèbre de Hopf d'arbres enracinés de Connes et Kreimer \cite{Connes1999,Connes2000}, l'algèbre de Hopf des arbres de Calaque, Ebrahimi-Fard et Manchon \cite{Calaque2011}, l'algèbre de Hopf des diagrammes orientés de Manchon \cite{Manchon2012}, à la théorie des champs quantiques non commutatifs \cite{Tanasa2013} ou encore la théorie des modèles de gravité quantique n-dimensionnels \cite{Markopoulou2003}\dots{} Dans leur article \cite{Duchamp2013}, Duchamp, Hoang-Nghia et Tanasa souhaitent construire un modèle d'algèbre de Hopf d'extraction-contraction pour les mots tassés et introduisent l'algèbre de Hopf $\WMat$. Outre sa construction, ils prouvent également qu'elle est librement engendrée par un ensemble appelé ensemble des irréductibles. Ils graduent cette nouvelle algèbre en fonction de la longueur des mots tassés, explicitent la dimension des espaces vectoriels engendrés par les éléments d'un degré donné et calculent le nombre de mots tassés irréductibles de degré fixé.

L'objectif de cet article est de mieux comprendre l'algèbre de Hopf $\WMat$, soit dans sa globalité, soit en considérant des sous-objets ou des objets quotients. 

Pour cela, nous nous intéressons dans une première section au coproduit et à l'antipode de $\WMat$ (propositions \ref{propantipodegen} et \ref{propantipode}). La non coliberté de cette algèbre de Hopf est établie et ses opérations sont ensuite dualisées afin de déterminer celles de son dual $\dual{\WMat}$. 

Dans une deuxième section, l'accent est mis sur des objets associés à $\WMat$. Nous débutons par une algèbre de Hopf de permutations cocommutative, notée $\SH$. Elle est à la fois un sous-objet et un objet quotient de $\WMat$. Son produit est toujours une concaténation décalée et son coproduit consiste essentiellement à partitionner en deux sous-ensembles les lettres d'un mots $w$ donné. Son dual $\dual{\SH}$ est colibre et commutatif. La structure d'algèbre est donnée par un double battage et la structure de cogèbre par une déconcaténation décalée. En se focalisant sur l'algèbre sous-jacente de $\dual{\SH}$, il est aisé de constater l'existence d'une structure de quadri-algèbre. Nous définissons par la suite une algèbre de Hopf de mots tassés stricts croissants, $\ISPW$. Cette dernière est construite en quotientant $\WMat$ par l'idéal de Hopf engendré par les mots tassés possédant la lettre $x_{0}$ dans leur écriture puis en considérant la sous-algèbre de Hopf librement engendrée par les classes des mots de la forme $\underbrace{x_{1}\dots x_{1}}_{n \text{ fois}}$ et a pour base les classes des mots tassés de la forme $\underbrace{x_{1}\dots x_{1}}_{n_{1} \text{ fois}}\ast\dots\ast\underbrace{x_{1}\dots x_{1}}_{n_{k} \text{ fois}}$. Le coproduit de la classe d'un mot de la forme $\underbrace{x_{1}\dots x_{1}}_{n_{1} \text{ fois}}\ast\dots\ast\underbrace{x_{1}\dots x_{1}}_{n_{k} \text{ fois}}$ est obtenu en partitionnant en deux sous-ensembles les $k$ mots de la forme $\underbrace{x_{1}\dots x_{1}}_{n_{i} \text{ fois}}$.  De par son caractère cocommutatif, $\ISPW$ s'écrit, d'après le \CQMM{}, comme algèbre enveloppante de ses primitifs. Nous déterminons alors l'existence d'une base particulière pour les éléments primitifs (proposition \ref{formebaseparticuliere}) ainsi que quelques éléments primitifs (propositions \ref{prim} à \ref{prim_Lambda}). L'étude de $\dual{\ISPW}$ permet d'affirmer qu'elle est isomorphe à l'algèbre de Hopf des fonctions quasi-symétriques (proposition \ref{qsymispwdual}) et donc, $\ISPW$ et l'algèbre de Hopf des fonctions symétriques non commutatives sont isomorphes (proposition \ref{nsymispw}). Nous terminons cette partie par l'introduction d'une algèbre de Hopf de compositions étendues notée $\Cp$. Elle s'obtient en identifiant chaque mot tassé $w$ avec chaque mot tassé $w_{\sigma}$ obtenu par permutation des lettres de $w$. On en donne une base dont les vecteurs sont de la forme $(\alpha_{0},\alpha_{1},\dots,\alpha_{n})$ où $\alpha_{0}$ est un entier naturel éventuellement nul et $\alpha_{1}$, \dots, $\alpha_{n}$ sont des entiers naturels non nuls. La composition étendue $(\alpha_{0},\alpha_{1},\dots,\alpha_{n})$ représente la classe du mot tassé $w_{\alpha}=\underbrace{x_{0}\dots x_{0}}_{\alpha_{1} \text{ fois}}\ast\dots\ast\underbrace{x_{1}\dots x_{1}}_{\alpha_{n} \text{ fois}}$. Ainsi, pour tout $i\in\llbracket0,n\rrbracket$, l'entier $\alpha_{i}$ correspond à la fréquence d'apparition $|w_{\alpha}|_{x_{\alpha_{i}}}$ de la lettre $x_{\alpha_{1}}$ dans le mot $w_{\alpha}$. Le produit de deux compositions étendues est obtenu par somme des fréquences d'apparition de la lettre $x_{0}$ et concaténation des fréquences des autres lettres. Le coproduit est obtenu à partir de $\De(w_{\alpha})$ dans $\WMat$. L'étude des deux opérations de $\Cp$ montre qu'elle est munie d'une structure de coproduit semi-direct d'algèbres de Hopf (propositions \ref{Hcomodule} à \ref{isomCoprodSemidirect}). Cette structure permet de mettre à jour deux actions de groupes (propositions \ref{actionDuale} et \ref{actionCaracteres}). Nous considérons ensuite les éléments primitifs de $\Cp$. Ils sont liés à ceux de $\ISPW$. Ceci fait l'objet des propositions \ref{lienCpISPW} à \ref{primCp}. 

L'objectif de la dernière section est de construire deux isomorphismes explicites; le premier entre $\dual{\ISPW}$ et $\QSym$, le deuxième entre $\ISPW$ et $\NSym$. Pour cela, on utilise un procédé introduit par Aguiar, Bergeron, Sottile dans \cite{Aguiar2006} exprimant $\QSym$ comme objet terminal dans la catégorie des algèbres de Hopf combinatoires. 

Vous trouverez une version plus détaillée et plus complète de l'ensemble de ces résultats dans mon manuscrit de thèse \cite{Mammezb}.

\section{Algèbre de Hopf des mots tassés}

Commençons par rappeler la construction de $\WMat$ \cite{Duchamp2013} ainsi que quelques notations utiles dans la suite de l'article. Rappelons également l'existence d'une famille de mots tassés qui engendrent librement l'algèbre de Hopf: les irréductibles. Par la suite, l'objectif est de mettre à jour quelques particularités de $\WMat$.  Nous établissons le caractère non-colibre de la cogèbre associée à $\WMat$ et explicitons l'antipode. Nous terminons par la présentation du dual gradué de $\WMat$ que l'on note $\dual{\WMat}$. 

\subsection{Algèbre de Hopf $\WMat$}

\subsubsection*{Notations}
Soit $\K$ un corps commutatif de caractéristique 0. On considère un alphabet infini, dénombrable, totalement ordonné $X=\{x_{0}<x_{1}<x_{2}<\dots\}$ et l'on note $X^{*}$ l'ensemble des mots formés à partir de $X$.\\
Pour tout entier naturel $n$ non nul, tout mot $w=x_{k_{1}}\dots x_{k_{n}}\in X^{*}$ et tous les entiers naturels $i$ et $s$, on pose :
\begin{align*}
|w|=& n,& |w|_{x_{i}}=&Card\big\{s\in\llbracket 1,n\rrbracket\text{ } |\text{ } k_{s}=i\big\},\\
\AAlph(w)=&\big\{x_{i}\text{ }|\text{ } |w|_{x_{i}}\neq 0\big\},& \IAlph(w)=&\big\{i\in\N\text{ }|\text{ } |w|_{x_{i}}\neq 0\big\},\\
\sup(w)=&\sup(\IAlph(w)),& T_{s}(w)=&x_{u_{1}}\dots x_{u_{n}} \text{ où } u_{j}=\begin{cases}k_{j}+s & \mbox{ si } k_{j}\neq 0, \\ 0 & \mbox{ sinon. }\end{cases}
\end{align*}
Si l'on pose $\IAlph(w)\setminus\{0\}=\{j_{1}<\dots<j_{s}\}$ (avec $s=\sup(w)$), on peut définir le mot tassé $\pack(w)= x_{p_{1}}\dots x_{p_{n}}$ où $p_{j}=\begin{cases} m & \mbox{ si } k_{j}=j_{m} \\ 0 & \mbox{ si } k_{j}=0 \end{cases}$ et $\PX=\big\{\pack(w)\text{ }|\text{ } w\in X^{*}\big\}$.
\Ex{} $\pack(x_{3}x_{7}x_{1}x_{8})=x_{2}x_{3}x_{1}x_{4}$ et $\pack(x_{50}x_{7}x_{0}x_{8})=x_{3}x_{1}x_{0}x_{2}$.\\

On définit le produit de concaténation décalée suivant: 
$$\ast:\left\{\begin{array}{rcl} \PX\ot\PX &\longrightarrow & \PX \\ (u,v) &\longrightarrow &
u\ast v = uT_{\sup(u)}(v).\end{array}\right.$$
\Ex{1} $x_{2}x_{1}x_{0}\ast x_{0}x_{1}x_{0}x_{3}x_{2}=x_{2}x_{1}x_{0}x_{0}x_{3}x_{0}x_{5}x_{4}$.\\

On définit le coproduit d'extraction-contraction suivant : 
$$\Delta:\left\{\begin{array}{rcl} \PX &\longrightarrow & \PX\otimes\PX \\ w &\longrightarrow &
\Delta(w)=\sum\limits_{I+J =\{1,\dots,|w|\}}\pack(w[I])\otimes \pack(w[J]/w[I])\end{array}\right.$$
où 
\begin{align*}
I+J& \text{ désigne l'union disjointe de } I \text{ et } J, \\
w[I]&= x_{k_{i_{1}}}\dots x_{k_{i_{l}}} \text{, si } I=\{i_{1}<\dots<i_{l}\},\\
w[J]/w[I]&= \tilde{x}_{k_{j_{1}}}\dots\tilde{x}_{k_{j_{|w|-l}}} \text{ où } J=\{j_{1}<\dots<j_{|w|-l}\} \text{ et } 
\tilde{x}_{k_{j_{t}}}=\begin{cases}
x_{k_{j_{t}}}& \mbox{ si } k_{j_{t}}\notin \AAlph(w[I]),\\
x_{0}& \mbox{ sinon. }
\end{cases}
\end{align*}

Pour une meilleure lisibilité, on définit le coproduit réduit par: 
$$\tilde{\Delta}:\left\{\begin{array}{rcl} \PX\setminus\K &\longrightarrow &  \PX\otimes \PX \\ w &\longrightarrow &
\tilde{\Delta}(w)=\Delta(w)-w\ot1-1\ot w.\end{array}\right.$$ Il est coassociatif, non unitaire. 

\Ex{}
\begin{align*}
\tilde{\De}(x_{1}x_{2}x_{0})
&=x_{1}\ot x_{1}x_{0}+x_{1}\ot x_{1}x_{0}+x_{0}\ot x_{1}x_{2}+x_{1}x_{2}\ot x_{0}+x_{1}x_{0}\ot x_{1}+x_{1}x_{0}\ot x_{1}\\
&=2x_{1}\ot x_{1}x_{0}+x_{0}\ot x_{1}x_{2}+x_{1}x_{2}\ot x_{0}+2x_{1}x_{0}\ot x_{1},\\
\tilde{\De}(x_{1}x_{2}x_{1})
&=x_{1}\ot x_{1}x_{0}+x_{1}\ot x_{1}x_{1}+x_{1}\ot x_{0}x_{1}+x_{1}x_{2}\ot x_{0}+x_{1}x_{1}\ot x_{1}+x_{2}x_{1}\ot x_{0}.
\end{align*}

L'espace $\WMat=(Vect(\PX),\ast,1_{X^{*}},\Delta,\varepsilon)$ (où $1_{X^{*}}$ et $\varepsilon$ désignent respectivement l'unité et la co-unité) est l'algèbre de Hopf des mots tassés sur laquelle nous allons travailler.
On rappelle (\emph{cf.} \cite[proposition 2]{Duchamp2013}) que $\WMat$ est librement engendrée par
$$\Irr(\WMat)=\bigg\{w\in\PX \text{ } | \text{ } \forall u,v\in\PX, (w=u\ast v)\Longrightarrow (u=1 \text{ ou } v=1)\bigg\}.$$
Tout mot $w\in \Irr(\WMat)$ est dit irréductible.\\

Dans la suite, sauf mention contraire, $w\in\WMat$ ou $w\in(\WMat)_{n}$  sous-entendra toujours que $w$ est un mot tassé ou un mot tassé de longueur $n$.

\subsubsection{Utilisation du coproduit pour prouver l'absence de coliberté}\label{WMatnoncol}
On souhaite montrer que $\WMat$ n'est pas colibre. Pour cela on considère les éléments primitifs de la cogèbre des mots tassés. Soit $F_{\WMat}$ la série formelle de $\WMat$.
On a, par calcul direct du nombre de mots tassés: $$F_{\WMat}(h)=1+2\sum_{\substack{n\in\N^{*} \\(\alpha_{1}\dots \alpha_{k})\models n}}\frac{n!}{\alpha_{1}!\dots\alpha_{k}!}h^{n}=1+2h+6h^{2}+26h^{3}+\dots$$
Si $\WMat$ était colibre la série formelle $F_{\Prim{\WMat}{}}$ de ses éléments primitifs serait égale à: $$F_{\Prim{\WMat}{}}(h)=1-\frac{1}{F_{\WMat}(h)}=2h+2h^{2}+10h^{3}+\dots$$
Cherchons alors une base de $\Prim{\WMat}{3}$, l'espace des éléments primitifs de degré 3, par calcul direct. On obtient : 
$$\Prim{\WMat}{3}=Vect(V_{1},\dots,V_{12})$$
où $V_{1},\dots,V_{12}$ sont les vecteurs linéairement indépendants suivants: 
\begin{align*}
V_{1}&=-x_{0}x_{1}x_{2}+2x_{1}x_{0}x_{2}-x_{1}x_{2}x_{0},& V_{2}&=-x_{0}x_{2}x_{1}+2x_{1}x_{0}x_{2}-2x_{1}x_{2}x_{0}+x_{2}x_{1}x_{0},\\
V_{3}&=x_{1}x_{0}x_{2}-x_{1}x_{2}x_{0}-x_{2}x_{0}x_{1}+x_{2}x_{1}x_{0},& V_{4}&=-x_{1}x_{2}x_{3}-2x_{3}x_{2}x_{1}+3x_{3}x_{1}x_{2},\\
V_{5}&=-x_{1}x_{3}x_{2}-x_{3}x_{2}x_{1}+2x_{3}x_{1}x_{2},& V_{6}&=-x_{2}x_{1}x_{3}-x_{3}x_{2}x_{1}+2x_{3}x_{1}x_{2},\\
V_{7}&=-x_{2}x_{3}x_{1}+x_{3}x_{1}x_{2},& V_{8}&=-x_{1}x_{2}x_{1}+x_{2}x_{1}x_{2},\\
V_{9}&=-x_{1}x_{2}x_{2}+2x_{2}x_{1}x_{2}-x_{2}x_{2}x_{1},& V_{10}&=-x_{2}x_{1}x_{1}+2x_{1}x_{2}x_{1}-x_{1}x_{1}x_{2},\\
V_{11}&=-x_{0}x_{1}x_{1}+2x_{1}x_{0}x_{1}-x_{1}x_{1}x_{0},& V_{12}&=-x_{1}x_{0}x_{0}+2x_{0}x_{1}x_{0}-x_{0}x_{0}x_{1}.\\
\end{align*}
On a alors: $$\dim(\Prim{\WMat}{3})=12\neq 10=\langle 1-\frac{1}{F_{\WMat}},h^{3}\rangle. $$
L'algèbre de Hopf $\WMat$ n'est donc pas colibre.

\subsubsection{Utilisation du produit et du coproduit dans le calcul de l'antipode}
Par itération du coproduit, pour tout entier $n\in\N^{*}$, tout mot $w\in(\WMat)_{n}$ et tout entier $k\in\llbracket 0,n-1\rrbracket$ on obtient:
\begin{equation}\label{coproditere}
 \Delta^{(k)}(w)=\sum\limits_{I_{1}+\dots+I_{k+1}=\{1,\dots,|w|\}}\pack(w[I_{1}])\otimes \pack(w[I_{2}]/w[I_{1}])\otimes\dots\otimes \pack(w[I_{k+1}]/w[I_{1}+\dots+I_{k}]).
\end{equation}
\begin{prop}\label{propantipodegen}
	Soit $n\in\N^{*}$ un entier et $w\in(\WMat)_{n}$ un mot tassé. En $w$, l'antipode vaut:
	$$S(w)=\sum_{k=1}^{|w|}(-1)^{k} \sum_{\substack{I_{1}+\dots+I_{k}=\{1,\dots,|w|\} \\ \forall j,~I_{j}\neq\emptyset}}  \pack(w[I_{1}]\ast w[I_{2}]/w[I_{1}]\ast\dots\ast w[I_{k}]/w[I_{1}+\dots+I_{k-1}]).$$
\end{prop}
\begin{proof}
	Soient $n\in\N^{*}$ un entier et $w\in(\WMat)_{n}$ un mot tassé. Pour démontrer la proposition énoncée on utilise le coproduit itéré explicité par la formule \ref{coproditere} ainsi que le fait que $S$ soit l'inverse de $Id$ pour le produit de convolution. On obtient alors:
	\begin{align*}
	S(w)&=-\sum_{\substack{I_{1}+I_{2}=\{1,\dots,|w|\}\\ I_{2}\neq\emptyset}}S(\pack(w[I_{1}]))\ast \pack(w[I_{2}]/w[I_{1}])\\
	&=\sum_{\substack{I_{1}+I_{2}+I_{3}=\{1,\dots,|w|\} \\ I_{2}\neq\emptyset\\ I_{3}\neq\emptyset}}S(\pack(w[I_{1}]))\ast \pack(w[I_{2}]/w[I_{1}])\ast \pack(w[I_{3}]/w[I_{1}+I_{2}])\\
	&=\sum_{k=1}^{|w|}(-1)^{k} \sum_{\substack{I_{1}+\dots+I_{k}=\{1,\dots,|w|\} \\ I_{j}\neq\emptyset}}  \pack(w[I_{1}])\ast\dots\ast \pack(w[I_{k}]/w[I_{1}+\dots+I_{k-1}])\\
	&=\sum_{k=1}^{|w|}(-1)^{k} \sum_{\substack{I_{1}+\dots+I_{k}=\{1,\dots,|w|\} \\ I_{j}\neq\emptyset}} \pack(w[I_{1}]\ast w[I_{2}]/w[I_{1}]\ast\dots\ast w[I_{k}]/w[I_{1}+\dots+I_{k-1}]).
	\end{align*}
	D'où le résultat énoncé.
\end{proof} 
L'image de certains mots tassés par l'antipode est remarquable comme l'illustre la proposition suivante. 
\begin{prop}\label{propantipode}
\begin{enumerate}
\item\label{antipodea} $S(\underbrace{x_{0}\dots x_{0}}_{n \text{ fois}})=(-1)^{n} \underbrace{x_{0}\dots x_{0}}_{n \text{ fois}}.$

\item\label{antipodeb} $S(\underbrace{x_{1}\dots x_{1}}_{n \text{ fois}})=(-1)^{n+1} \sum\limits_{k=1}^{n}\displaystyle\binom{n}{k} (-1)^{k} \underbrace{x_{1}\dots x_{1}}_{k \text{ fois}}\underbrace{x_{0}\dots x_{0}}_{n-k \text{ fois}}.$

\item\label{antipodec} Soit $(\alpha_{1},\dots,\alpha_{n})\in(\N^{*})^{n}$. On a : $$S(\underbrace{x_{1}\dots x_{1}}_{\alpha_{n} \text{ fois}}\dots\underbrace{x_{n}\dots x_{n}}_{\alpha_{1} \text{ fois}}) =(-1)^{(n+\sum\limits_{j=1}^{n}\alpha_{j})}\sum_{s=n}^{\alpha_{1}+\dots+\alpha_{n}}\sum_{\substack{k_{1}+\dots+k_{n}=s\\1\leq k_{j}\leq \alpha_{j}}}(-1)^{s}\binom{\alpha_{1}}{k_{1}}\dots\binom{\alpha_{n}}{k_{n}}\omega_{\alpha,k}$$
où 
$$\omega_{\alpha,k}=\underbrace{x_{1}\dots x_{1}}_{k_{1} \text{ fois}}\underbrace{x_{0}\dots x_{0}}_{\alpha_{1}-k_{1} \text{ fois}}\dots\underbrace{x_{n}\dots x_{n}}_{k_{n} \text{ fois}}\underbrace{x_{0}\dots x_{0}}_{\alpha_{n}-k_{n} \text{ fois}}.$$

\item\label{antipoded} $S(x_{1}\dots x_{n})=(-1)^{n}x_{1}\dots x_{n}.$

\item\label{antipodee} $S(x_{n}\dots x_{1})=\sum\limits_{k=1}^{n}(-1)^{k}\sum\limits_{\alpha=(\alpha_{1},\dots,\alpha_{k})\models n} \frac{n!}{\alpha_{1}!\dots\alpha_{k}!}w_{n,\alpha}$\\
où  $$w_{n,\alpha}=x_{\alpha_{1}}\dots x_{1}x_{(\alpha_{1}+\alpha_{2})}\dots x_{(\alpha_{1}+1)}\dots x_{(\alpha_{1}+\dots+\alpha_{k})}\dots x_{(\alpha_{1}+\dots+\alpha_{k-1}+1)}.$$

\item\label{antipodef} $\forall i\in\llbracket 1,n\rrbracket$ on a : $$S(x_{i}\dots x_{1}x_{i+1}\dots x_{n})=\sum\limits_{k=1}^{i}\sum\limits_{\alpha=(\alpha_{1},\dots,\alpha_{k})\models i}(-1)^{(n+i+k)} \frac{i!}{\alpha_{1}!\dots\alpha_{k}!} x_{1}\dots x_{n-i}T_{n-i}(w_{n,\alpha})$$
où $$T_{n-i}(w_{n,\alpha})=x_{(\alpha_{1}+n-i)}\dots x_{(1+n-i)}\dots x_{(\alpha_{1}+\dots+\alpha_{k}+n-i)}\dots x_{(\alpha_{1}+\dots+\alpha_{k-1}+1+n-i)}.$$

\item\label{antipodeg} $\forall i\in\llbracket 1,n\rrbracket$ on a : $$S(x_{1}\dots x_{n-i}x_{n}\dots x_{n-i+1})=\sum\limits_{k=1}^{i}\sum\limits_{\alpha=(\alpha_{1},\dots,\alpha_{k})\models i}(-1)^{(n+i+k)} \frac{i!}{\alpha_{1}!\dots\alpha_{k}!} w_{i,\alpha}x_{i+1}\dots x_{n}.$$
\end{enumerate}
\end{prop}

\begin{proof}
	Pour démontrer les points \ref{antipodea}, \ref{antipodeb} et \ref{antipodee}, il suffit d'effectuer une récurrence sur la longueur du mot en considérant le fait que $S$ est l'inverse de $Id$ pour le produit de convolution. Les points \ref{antipodec}, \ref{antipodee} et \ref{antipodeg} sont montrés grâce au fait que l'antipode soit un anti-morphisme d'algèbres. La formule \ref{antipoded} n'est que le cas particulier $(\alpha_{1},\dots,\alpha_{n})=(1,\dots,1)$ de la formule \ref{antipodec}.
\end{proof}

\Ex{}
\begin{align*}
	S(x_{1}x_{1}x_{2}x_{2}x_{2})=&x_{1}x_{1}x_{1}x_{2}x_{2}-2x_{1}x_{1}x_{1}x_{2}x_{0}-3x_{1}x_{1}x_{0}x_{2}x_{2}+6x_{1}x_{1}x_{0}x_{2}x_{0}\\
	+&3x_{1}x_{0}x_{0}x_{2}x_{2}-6x_{1}x_{0}x_{0}x_{2}x_{0},\\
	S(x_{2}x_{1}x_{3}x_{4})=&-x_{1}x_{2}x_{4}x_{3}+2x_{1}x_{2}x_{3}x_{4},\\
	S(x_{1}x_{2}x_{4}x_{3})=&-x_{2}x_{1}x_{3}x_{4}+2x_{1}x_{2}x_{3}x_{4}.
\end{align*} 

\subsection{Algèbre de Hopf duale $\dual{\WMat}$}\label{operationsdualWMat}
L'algèbre de Hopf $\WMat$ étant graduée connexe, son dual gradué est une algèbre de Hopf que l'on notera $\dual{\WMat}$. On considère $(Z_{w})_{w\in \WMat}$ la base duale des mots tassés. Par transposition des opérations de $\WMat$ on obtient celles de $\dual{\WMat}$. Commençons par la structure la plus facile à décrire, à savoir celle de cogèbre.
\begin{prop}
	Le coproduit de $\dual{\WMat}$ est l'application suivante:
	$$\De:\left\{\begin{array}{rcl} \dual{\WMat} &\longrightarrow & \dual{\WMat}\ot\dual{\WMat}  \\ 
	Z_{w} &\longrightarrow & \Delta(Z_{w})=\sum\limits_{i=0}^{k}Z_{w_{1}\ast\dots\ast w_{i}}\otimes Z_{w_{i+1}\ast\dots\ast w_{k}} \end{array}\right.$$
	où $w_{1}\ast\dots\ast w_{k}$ est l'unique décomposition de $w$ en produit d'éléments irréductibles de $\WMat$. 
\end{prop} 
\begin{proof}
	Soient $w, m_{1}, m_{2}\in \WMat$ des mots tassés. On a :
	\begin{align*}
	\Delta(Z_{w})(m_{1}\otimes w_{2})&=(Z_{w}\circ\ast)(m_{1}\otimes m_{2})=Z_{w}(m_{1}\ast m_{2})\\
	&=Z_{w}(m_{1}T_{\sup(m_{1})}m_{2})=\begin{cases} 
	1 & \mbox{ si } w=m_{1}\ast m_{2},\\
	0 & \mbox{ sinon. } 
	\end{cases}
	\end{align*}
	Or, d'après \cite{Duchamp2013}, on sait que : 
	$$\text{il existe un unique } k\in\N \text{ et un unique } (w_{1},\dots,w_{k})\in \Irr(\WMat)^{k} \text{ tels que } w=w_{1}\ast\dots\ast w_{k}.$$
	Donc $$\Delta(Z_{w})=\sum\limits_{i=0}^{k}Z_{w_{1}\ast\dots\ast w_{i}}\otimes Z_{w_{i+1}\ast\dots\ast w_{k}}.$$ 
\end{proof}
	\Ex{}
	\begin{align*}
	\De(Z_{x_{2}x_{1}x_{3}})&=1\ot Z_{x_{2}x_{1}x_{3}}+Z_{x_{2}x_{1}}\ot Z_{x_{1}}+ Z_{x_{2}x_{1}x_{3}}\ot1,\\ \De(Z_{x_{2}x_{1}x_{2}})&=Z_{x_{2}x_{1}x_{2}}\ot 1 + 1\ot Z_{x_{2}x_{1}x_{2}}.
	\end{align*}
	
Examinons maintenant la structure d'algèbre de $\dual{\WMat}$.
Soient $w_{1}, w_{2}\in\WMat$ des mots tassés. On définit les conditions suivantes:
\begin{align*}
	C_{1}~:~& x_{0}\notin \AAlph(w_{1})\cup \AAlph(w_{2}),\\
	C_{2}~:~& \bigg(\AAlph(w_{1})\setminus\{x_{0}\}\bigg)\cap \bigg(\AAlph(w_{2})\setminus\{x_{0}\}\bigg)\neq\emptyset \mbox{ et } x_{0}\in \AAlph(w_{1})\cup \AAlph(w_{2}), \\
	C_{3}~:~& w_{1}=\underbrace{x_{0}\dots x_{0}}_{n_{1} \text{ fois}}, \\
	C_{4}~:~& w_{2}=\underbrace{x_{0}\dots x_{0}}_{n_{2} \text{ fois}}. 
\end{align*}

Définissons maintenant la structure d'algèbre de $\dual{\WMat}$.
\begin{prop}
	Le produit de $\dual{\WMat}$ est l'application suivante:
	$$m:\left\{\begin{array}{rcl} \dual{\WMat}\ot\dual{\WMat} &\longrightarrow & \dual{\WMat}  \\ 
	Z_{w_{1}}\ot Z_{w_{2}} &\longrightarrow & Z_{w_{1}}Z_{w_{2}} \end{array}\right.$$
	telle que $$Z_{w_{1}}Z_{w_{2}}= \begin{cases} 
		\sum\limits_{\substack{\tau\in Bat(s_{1},s_{2}), \\\mu\in Bat(n_{1},n_{2})}}Z_{\tau\circ(w_{1}\ast w_{2})\circ\mu^{-1}} & \mbox{ si } C_{1},\\
		\sum\limits_{\substack{\tau\in Bat(s_{1},s_{2}), \\ \mu\in Bat(n_{1},n_{2})}}Z_{\tau\circ(w_{1}\tilde{w_{2}})\circ\mu^{-1}} & \mbox{ si } C_{2}, \\
		\sum\limits_{\mu\in Bat(n_{1},n_{2})}Z_{(w_{1}\ast w_{2})\circ\mu^{-1}} & \mbox{ si } C_{3}, \\
		\sum\limits_{\mu\in Bat(n_{1},n_{2})}Z_{(w_{1}\tilde{w_{2}})\circ\mu^{-1}} & \mbox{ si } C_{4},
	\end{cases}$$
	où pour $i\in\{1,2\}$, 
	\begin{align*}
	n_{i}=&|w_{i}|,\\
	s_{i}=&\sup(w_{i}),\\
	Bat(n_{1},n_{2})=&\{\sigma\in\Sn{n_{1}+n_{2}}\text{, } \sigma(1)<\dots<\sigma(n_{1}) \text{ et } \sigma(n_{1}+1)<\dots<\sigma(n_{1}+n_{2})\},
	\end{align*}
	$$\tilde{w_{2}}=\sum_{w\in\Gamma_{(w_{1},w_{2})}}w$$
	avec \begin{align*}
	&\Gamma_{(w_{1},w_{2})} = \\
	&\Bigg\{ w\in\WMat \text{, } |w|=n_{2} \text{, } w[\{i\}]= \begin{cases} 
	T_{s_{1}}(w_{2}[\{i\}]) & \mbox{ si } w_{2}[\{i\}]\neq x_{0}, \\
	x_{j} \mbox{ avec } j\in \IAlph(w_{1})\cup\{0\} & \mbox{ sinon} 
	\end{cases} \Bigg\}.
	\end{align*}
	et $w_{1}\tilde{w_{2}}$ désigne la concaténation de $w_{1}$ et de $\tilde{w_{2}}$.
\end{prop}
\begin{proof}
	 Soient $w_{1}, w_{2},w\in\WMat$ des mots tassés. Pour $i\in\{1,2\}$, on pose: $n_{i}=|w_{i}|$ et $s_{i}=\sup(w_{i})$. 
	 On a alors:\begin{align*}
			Z_{w_{1}}Z_{w_{2}}(w)&=(Z_{w_{1}}\ot Z_{w_{2}})\circ\De(m)\\
			&=\sum_{I+J =\{1,\dots,|w|\}}Z_{w_{1}}(\pack(w[I]))Z_{w_{2}}(\pack(w[J]/w[I]))\\
			&=\sum_{\substack{I+J =\{1,\dots,|w|\},\\ |I|=n_{1},\\ |J|=n_{2}}}Z_{w_{1}}(\pack(w[I]))Z_{w_{2}}(\pack(w[J]/w[I]))\\
			&=\sum_{\substack{I+J =\{1,\dots,|w|\},\\ |I|=n_{1},\\ |J|=n_{2}}}\delta_{w_{1},\pack(w[I])}\delta_{w_{2},\pack(w[J]/w[I])}\\
			&= \begin{cases} 
			\sum\limits_{\substack{\tau\in Bat(s_{1},s_{2}), \\\mu\in Bat(n_{1},n_{2})}}Z_{\tau\circ(w_{1}\ast w_{2})\circ\mu^{-1}}(m) & \mbox{ si } C_{1},\\
			\sum\limits_{\substack{\tau\in Bat(s_{1},s_{2}), \\ \mu\in Bat(n_{1},n_{2})}}Z_{\tau\circ(w_{1}\tilde{w_{2}})\circ\mu^{-1}}(m) & \mbox{ si } C_{2}, \\
			\sum\limits_{\mu\in Bat(n_{1},n_{2})}Z_{(w_{1}\ast w_{2})\circ\mu^{-1}}(m) & \mbox{ si } C_{3}, \\
			\sum\limits_{\mu\in Bat(n_{1},n_{2})}Z_{(w_{1}\tilde{w_{2}})\circ\mu^{-1}}(m) & \mbox{ si } C_{4}.
			\end{cases}
	 \end{align*}
\end{proof}
\Ex{} Voici quelques exemples de produits d'éléments de $\dual{\WMat}$. On illustre toutes les configurations possibles. On les présente dans l'ordre suivant: deux mots vérifiant la condition $C_{1}$, puis la condition $C_{3}$, puis la condition $C_{4}$ et enfin, la  condition la plus complexe, $C_{2}$. Dans le dernier exemple, pour une meilleure lisibilité, on ne développe pas le calcul du fait du nombre de termes dans les produits de battage.
 \begin{align*}
 Z_{x_{1}}Z_{x_{1}x_{1}}&=Z_{x_{1}\shuffle x_{2}x_{2}}+Z_{x_{2}\shuffle x_{1}x_{1}}\\
 &=Z_{x_{1}x_{2}x_{2}}+Z_{x_{2}x_{1}x_{2}}+Z_{x_{2}x_{2}x_{1}}+Z_{x_{2}x_{1}x_{1}}+Z_{x_{1}x_{2}x_{1}}+Z_{x_{1}x_{1}x_{2}},\\
 Z_{x_{0}}Z_{x_{1}}&=Z_{x_{0}\shuffle x_{1}}=Z_{x_{1}x_{0}}+Z_{x_{0}x_{1}},\\
 Z_{x_{1}}Z_{x_{0}}&=Z_{x_{1}\shuffle x_{0}}+Z_{x_{1}\shuffle x_{1}}=Z_{x_{1}x_{0}}+Z_{x_{0}x_{1}}+2Z_{x_{1}x_{1}},\\
 Z_{x_{1}x_{0}}Z_{x_{0}x_{1}}&=Z_{x_{1}x_{0}\shuffle x_{0}x_{2}}+Z_{x_{1}x_{0}\shuffle x_{1}x_{2}}+Z_{x_{2}x_{0}\shuffle x_{0}x_{1}}+Z_{x_{2}x_{0}\shuffle x_{2}x_{1}}.
\end{align*}

Après avoir considéré les algèbres de Hopf $\WMat$ et $\dual{\WMat}$ dans leur globalité, intéressons-nous maintenant à leur structure interne et tirons-en de nouvelles algèbres de Hopf.

\section{Quelques algèbres de Hopf associées}

Dans cette section, on s'intéresse à la mise en évidence d'algèbres de Hopf liées à $\WMat$ en tant que sous-algèbre de Hopf ou algèbre de Hopf quotient.

On commence par considérer l'espace vectoriel engendré par les mots tassés issus de permutations (ex: $x_{2}x_{1}$, $x_{3}x_{2}x_{4}x_{1}$). Il s'agit en réalité d'une sous-algèbre de Hopf de $\WMat$, notée $\SH$. Elle peut également être vue comme algèbre de Hopf quotient de $\WMat$. Par dualité, on obtient une algèbre de Hopf, à la fois quotient et sous-objet de $\dual{\WMat}$, possédant une structure de quadri-algèbre ainsi qu'une double structure d'algèbre dendriforme.  

On considère par la suite l'espace vectoriel engendré par les mots tassés ne contenant pas de lettres d'indice nul. Il forme une sous-algèbre de $\WMat$ mais n'en est pas une sous-cogèbre. Il est toutefois possible de l'interpréter comme une algèbre de Hopf quotient. On s'intéresse à son sous-objet des mots tassés stricts croissants, $\ISPW$. On en détermine des éléments primitifs. On explicite également les opérations de $\dual{\ISPW}$. Ceci permet ainsi d'établir que $\dual{\ISPW}$ est isomorphe à l'algèbre de Hopf des fonctions quasi-symétriques notées $\QSym$. Ainsi $\ISPW$ est isomorphe à l'algèbre de Hopf des fonctions symétriques non commutatives notées $\NSym$.

Le troisième cas d'étude concerne l'espace vectoriel engendré par les mots tassés croissants pouvant posséder des lettres d'indice nul. Dans ce cas, il ne s'agit ni d'une sous-algèbre ni d'une sous-cogèbre, mais il est interprétable comme algèbre de Hopf quotient. On décide de considérer cette dernière sous sa forme d'algèbre de Hopf de compositions étendues $\Cp$. Elle peut être décrite comme coproduit semi-direct d'une algèbre de Hopf tensorielle $C=\tens{(0,n), n\in\N^{*}}$ par une algèbre de Hopf de polynôme $H=\K[(1)]$. Cette construction met en évidence une coaction $\rho$ de $C$ dans $C\ot H$. Cette coaction permet, par transposition, d'obtenir une action de $\dual{H}$ sur $\dual{C}$. Elle fournit également une action du groupe des caractères $Char(H)$ de $H$ sur celui des caractères $Char(C)$ de $C$. On s'intéresse par la suite aux éléments primitifs de $\Cp$. Ils sont, d'une certaine façon, liés à ceux de $\ISPW$. On peut ainsi donner les éléments primitifs de $\Cp$ jusqu'en degré 7. On explicite également les opérations de $\dual{\Cp}$.

\subsection{Algèbre de Hopf de permutations}
\subsubsection{Algèbre de Hopf $\SH$} 
On s'intéresse au sous-espace vectoriel de $\WMat$ suivant: $$\SH= Vect\bigg(w\in\WMat, \IAlph(w)=\{1,\dots,|w|\}\bigg)=Vect\bigg(\underbrace{x_{\sigma(1)}\dots x_{\sigma(n)}}_{\text{noté } w_{\sigma}}, n\in\N, \sigma\in\Sn{n} \bigg).$$
Pour des raisons pratiques on identifiera parfois les éléments de $\SH$ et les permutations, c'est-à-dire : $$w_{\sigma}=x_{\sigma(1)}\dots x_{\sigma(n)}=\sigma(1)\dots\sigma(n)=\begin{pmatrix}1 & \dots & n \\ \sigma(1) & \dots & \sigma(n) \end{pmatrix}=\sigma.$$
Par calcul direct, on obtient que $\SH$ est une sous-algèbre de Hopf de $\WMat$.
\ExCex
$$x_{1}x_{2}x_{3}\in\SH,~x_{3}x_{1}x_{2}+8x_{1}x_{2}\in\SH,~x_{1}x_{1}x_{2}+7x_{2}x_{1}x_{3}\notin\SH,~ x_{5}x_{0}x_{3}x_{1}x_{4}x_{2}-x_{1}\notin\SH.$$

\begin{prop}
	$\SH$ est librement engendrée par $\SH\cap \Irr(\WMat)$. 
\end{prop}
\begin{proof}
	Considérons $\sigma\in \Sn{n}$ et $w_{\sigma}=x_{\sigma_{1}}\dots x_{\sigma_{n}}$.\\
	Alors: $$\text{il existe un unique }k\in\N \text{ et un unique } (w_{1},\dots,w_{k})\in \Irr(\WMat)^{k} \text{ tels que } w=w_{1}\ast\dots\ast w_{k}.$$
	Supposons qu'il existe un entier $i_{0}\in\{1,\dots,k\}$ tel que $w_{i_{0}}\notin \Sn{|w_{i_{0}}|}$. On a alors deux possibilités:
	\begin{enumerate}
		\item soit $x_{0}\in \AAlph(w_{i_{0}})$ et alors $x_{0}\in \AAlph(w_{\sigma})$. Ceci est exclu.
		\item soit il existe un entier $s\in\{1,\dots,|w_{i_{0}}|\}$, tel que $|w_{i_{0}}|_{x_{s}}\geq 2$. Cela implique alors que le mot $w_{\sigma}$ contient au moins deux lettres d'indice $s+\sum\limits_{u=1}^{i_{0}-1}\sup(w_{u})$. Ceci est exclu.
	\end{enumerate}
	On obtient donc que pour tout entier $i\in\{1,\dots,k\}$, $w_{i}\in \Sn{|w_{i}|}$.	
\end{proof}

Décrivons le produit et le coproduit adaptés à la sous-algèbre de Hopf $\SH$. En ce qui concerne le produit, il s'écrit comme restriction de celui de $\WMat$, \emph{i.e.}:
$$\ast:\left\{\begin{array}{rcl} \SH\ot\SH &\longrightarrow & \SH \\ w_{\sigma}\ot w_{\tau} &\longrightarrow &
w_{\sigma}\ast w_{\tau} = w_{\sigma}T_{|w_{\sigma}|}(w_{\tau}).\end{array}\right.$$
\Ex{}
\begin{align*}
	x_{2}x_{1}x_{3}\ast x_{5}x_{1}x_{4}x_{3}x_{2}=&x_{2}x_{1}x_{3}x_{8}x_{4}x_{7}x_{6}x_{5},\\
	x_{5}x_{1}x_{4}x_{3}x_{2}\ast x_{2}x_{1}x_{3}=&x_{5}x_{1}x_{4}x_{3}x_{2}x_{7}x_{6}x_{8}.
\end{align*}
En ce qui concerne le coproduit de $\SH$, la restriction de celui de $\WMat$ s'écrit sous la forme:
$$\De:\left\{\begin{array}{rcl} \SH &\longrightarrow &  \SH\ot\SH \\ w_{\sigma} &\longrightarrow &
\Delta(w_{\sigma})=\sum\limits_{I + J=\{1,\dots,n\}}\pack(w[I])\otimes \pack(w[J]).\end{array}\right.$$
\Ex{1} Pour une meilleure lisibilité, on ne donne que le coproduit réduit dans l'exemple suivant.
$$\tilde{\De}(x_{3}x_{1}x_{2})=2x_{1}\otimes x_{2}x_{1}+x_{1}\otimes x_{1}x_{2}+2x_{2}x_{1}\otimes x_{1}+x_{1}x_{2}\otimes x_{1}.$$

\begin{prop}\label{SHquotient}
	L'algèbre de Hopf $\SH$ peut être vue comme une algèbre de Hopf quotient de $\WMat$.
\end{prop}
\begin{proof}
Il suffit de considérer la surjection canonique $\Pi_{\SH} : \WMat \longrightarrow \SH$. Son noyau est un biidéal de Hopf.
\end{proof}

\subsubsection{Algèbre de Hopf $\dual{\SH}$}
D'après la proposition \ref{SHquotient}, $\dual{\SH}$ est à la fois une algèbre de Hopf quotient et une sous-algèbre de Hopf de $\dual{\WMat}$. Pour la suite, on choisira de considérer $\dual{\SH}$ comme un sous-objet de $\dual{\WMat}$. On notera la base duale des mots tassés issus de permutations soit par $(Z_{w_{\sigma}})_{w_{\sigma}\in \SH}$, soit par $(Z_{\sigma})_{\substack{\sigma\in\Sn{n} \\n\in\N^{*} }}$ où,
$Z_{w_{\sigma}}=Z_{\sigma}$ suivant ce qui est le plus adapté.
\paragraph{Description des opérations de $\dual{\SH}$.} 
Le coproduit de $\dual{\SH}$ est la restriction de celui de $\dual{\WMat}$ à $\dual{\SH}$.
\Ex{1} On considère le mot tassé $x_{3}x_{1}x_{2}x_{5}x_{4}x_{6}=x_{3}x_{1}x_{2}\ast x_{2}x_{1}\ast x_{1}$. Son coproduit réduit vaut:
$$\tilde{\De}(Z_{x_{3}x_{1}x_{2}x_{5}x_{4}x_{6}})=Z_{x_{3}x_{1}x_{2}}\ot Z_{x_{2}x_{1}x_{3}}+Z_{x_{3}x_{1}x_{2}x_{5}x_{4}}\ot Z_{x_{1}}.$$

Le plus intéressant concerne le produit de $\dual{\SH}$. En effet, comme $\SH$ est cocommutative, $\dual{\SH}$ est commutative. De plus, le produit de $\dual{\SH}$ est la restriction à $\dual{\SH}$ de celui de $\dual{\WMat}$, ce qui revient à ne considérer que la condition $C_{1}$ décrite dans \ref{operationsdualWMat}. Il s'écrit donc:  
	$$m:\left\{\begin{array}{rcl} \dual{\SH}\ot\dual{\SH} &\longrightarrow & \dual{\SH}  \\ 
	Z_{\sigma_{1}}\ot Z_{\sigma_{2}} &\longrightarrow & Z_{\sigma_{1}}Z_{\sigma_{2}}=\sum\limits_{\tau,\mu\in Bat(n_{1},n_{2})}Z_{\tau\circ(\sigma_{1}\ast \sigma_{2})\circ\mu^{-1}} \end{array}\right.$$
où $n_{1}$ et $n_{2}$ sont les entiers naturels non nuls tels que $\sigma_{1}\in\Sn{n_{1}}$ et $\sigma_{2}\in\Sn{n_{2}}$.
\Ex{1}
\begin{align*}
Z_{x_{1}}Z_{x_{2}x_{1}}&=Z_{x_{1}\shuffle x_{3}x_{2}}+Z_{x_{2}\shuffle x_{3}x_{1}}+Z_{x_{3}\shuffle x_{2}x_{1}}\\
&=3Z_{x_{3}x_{2}x_{1}}+Z_{x_{1}x_{3}x_{2}}+2Z_{x_{3}x_{1}x_{2}}+2Z_{x_{2}x_{3}x_{1}}+Z_{x_{2}x_{1}x_{3}},\\
\end{align*}
\Rq{1} La structure d'algèbre de Hopf de $\dual{\SH}$ est la même que celle construite par Aguiar et Mahajan  à partir du monoïde de Hopf des ordres linéaires et d'un foncteur de Fock \cite[Partie III, chapitre 15, exemple 15.17]{Aguiar2010}. On retrouve également cette structure dans les travaux de Vargas \cite{Vargas2014} sous le nom d'algèbre de Hopf de super-battages. 

\paragraph{Structure de quadri-algèbre et d'algèbre dendriforme sur $\dual{\SH}$.}
\subparagraph{Rappels.} 
Commençons par quelques rappels sur les quadri-algèbres, les algèbres dendriformes et les algèbres de Zinbiel aussi appelées algèbres dendriformes commutatives ou encore algèbres de Leibniz duales. 

Les algèbres dendriformes ont été introduites par Loday \cite{Loday2001}. Elles permettent d'écrire certains produits associatifs comme somme d'un produit gauche et d'un produit droit. Cela permet donc de casser l'associativité et d'obtenir des relations de compatibilité souvent plus faciles à vérifier. Dans \cite{Loday1998}, Loday et Ronco, grâce à une structure dendriforme, définissent une algèbre de Hopf d'arbres binaires appelée algèbre de Loday et Ronco. Elle est librement engendrée par l'arbre binaire à un seul sommet interne. Foissy démontre \cite[proposition 31]{Foissy2002b} que l'algèbre de Hopf de Loday et Ronco décorée est isomorphe à l'algèbre de Hopf des arbres enracinés plans décorés. Cette démonstration repose sur les caractères dendriformes des deux objets. Aguiar et Sottile \cite{Aguiar2006a} étudient le dual gradué de l'algèbre de Loday et Ronco. La structure bidendriforme est introduite par Foissy \cite{Foissy2007}. En plus de scinder l'associativité, elle scinde également la coassociativité. Foissy explicite également la structure bidendriforme de l'algèbre de Malvenuto et Reutenauer $\FQSym$ ainsi que celle de l'algèbre de Hopf des arbres enracinés plans décorés de Connes et Kreimer.  Des versions analogues au \CQMM{} sont prouvées: par Ronco \cite{Ronco2001} pour les algèbres dendriformes, par Chapoton \cite{Chapoton2002} pour les bigèbres dendriformes et par Foissy \cite{Foissy2007} pour les bigèbres bidendriformes. Le cas bidendriforme implique \cite[théorème 39]{Foissy2007} que $\FQSym$ est isomorphe à une algèbre de Hopf d'arbres enracinés plans décorés.

Les algèbres de Zinbiel ont été définies par Loday \cite{Loday1995,Loday2001}. Elles sont équivalentes aux algèbres dendriformes commutatives et sont liées aux algèbres de Leibniz (les opérades $\Zinb$ et $\Lieb$ son duales dans la dualité de Koszul).  Livernet \cite{Livernet1998} remplace les algèbres de Lie et les algèbres commutatives par les algèbres de Leibniz et les algèbres de Zinbiel pour définir une théorie non commutative des homotopies rationnelles. Chapoton \cite[théorème 5.9]{Chapoton2005} démontre que l'opérade $\Zinb$ est anticyclique (il existe une action du groupe symétrique $\Sn{n+1}$ sur le $\Sn{n}$-module $\mathcal{P}(n)$ pour tout entier naturel $n$). Vallette \cite[théorème 20]{Vallette2008} donne une description de $\Zinb$ en fonction des opérades $\PreLie$ et $\Com$ grâce au produit de Manin.

En considérant une algèbre dendriforme et en scindant les produits gauche et droit en deux lorsque cela est possible, on obtient une quadri-algèbre. Aguiar et Loday introduisent cette notion dans \cite{Aguiar2004}. Ils déterminent une structure de quadri-algèbre sur les algèbres infinitésimales et s'intéressent à la quadri-algèbre libre à un générateur \cite[corollaire 3.3, section 4]{Aguiar2004}.  Vallette \cite[corollaires 48 et 52, théorème 51]{Vallette2008} prouve les conjectures émises par Aguiar et Loday dans \cite[conjectures 4.2, 4.5 et 4.6]{Aguiar2004} concernant les opérades $\Quad$ et $\Quad^{!}$. Foissy présente la quadri-algèbre libre à un générateur comme sous objet de $\FQSym$ \cite[corollaire 7]{Foissy2015} et définit la notion de quadri-bigèbre en détaillant l'exemple de $\FQSym$ \cite[section 3]{Foissy2015}.

\begin{defi}
	Une quadri-algèbre $\mathcal{Q}$ est un espace vectoriel muni de quatre opérations $\searrow$, $\nearrow$, $\nwarrow$ et $\swarrow$, définies sur $\mathcal{Q}\ot\mathcal{Q}$, à valeurs dans $\mathcal{Q}$, et, vérifiant les axiomes suivants: pour tout $x,y,z\in\mathcal{Q}$,
	\begin{align*}
		(x\nwarrow y)\nwarrow z&=x\nwarrow(y\cdot z), & (x\nearrow y)\nwarrow z&=x\nearrow(y\prec z), & (x\wedge y)\nearrow z&=x\nearrow(y\succ z),\\
		(x\swarrow y)\nwarrow z&=x\swarrow(y\wedge z), & (x\searrow y)\nwarrow z&=x\searrow(y\nwarrow z), & (x\vee y)\nearrow z&=x\searrow(y\nearrow z),\\
		(x\prec y)\swarrow z&=x\swarrow(y\vee z), & (x\succ y)\swarrow z&=x\searrow(y\swarrow z), & (x\cdot y)\searrow z&=x\searrow(y\searrow z),
	\end{align*}
	où, pour tout $x,y\in\mathcal{Q}$, \begin{align*}
		x\prec y&=x\nwarrow y+x\swarrow y,& x\wedge y&=x\nearrow y+x\nwarrow y,\\
		x\succ y&=x\nearrow y+x\searrow y,& x\vee y&=x\searrow y+x\swarrow y,\\
	\end{align*}
	et $$x\cdot y= x\nwarrow y+x\swarrow y+x\nearrow y+x\searrow y=x\prec y+x\succ y=x\wedge y+x\vee y.$$
\end{defi}
\begin{defi}
	Une algèbre dendriforme est un espace vectoriel $\mathcal{D}$ muni de deux opérations, un produit gauche $\prec:\mathcal{D}\ot\mathcal{D}\longrightarrow \mathcal{D}$ et un produit droit $\succ:\mathcal{D}\ot\mathcal{D}\longrightarrow \mathcal{D}$, vérifiant les axiomes suivants: $\forall x,y,z\in\mathcal{D}$,
	\begin{align*}
	(x\prec y)\prec z&=x\prec(y\prec z)+x\prec(y\succ z),\\
	(x\succ y)\prec z&=x\succ(y\prec z),\\
	(x\prec y)\succ z+ (x\succ y)\succ z&=x\succ(y\succ z).
	\end{align*}
\end{defi}

\begin{prop}
	Soit $(\mathcal{Q},\searrow, \nearrow, \nwarrow ,\swarrow)$ une quadri-algèbre alors, $(\mathcal{Q},\nwarrow+\swarrow,\nearrow+ \searrow)$, \emph{i.e.}  $(\mathcal{Q},\prec,\succ)$, et $(\mathcal{Q},\nearrow+\nwarrow, \searrow+\swarrow)$, \emph{i.e.} $(\mathcal{Q},\wedge, \vee)$, sont des algèbres dendriformes.
\end{prop}
\begin{defi}
	Une algèbre de Zinbiel est un espace vectoriel $\mathcal{Z}$ muni d'une opération binaire définie sur $\mathcal{Z}\ot\mathcal{Z}$, à valeur dans $\mathcal{Z}$ telle que: $$\text{pour tout } x,y,z\in\mathcal{Z},~(x\cdot y)\cdot z=x\cdot(y\cdot z)+x\cdot (z\cdot y).$$
\end{defi}
\begin{prop}
	Soit $(\mathcal{Z},\cdot)$ une algèbre de Zinbiel. Pour tout couple $(x,y)\in Z^{2}$ on pose $x\prec y=x\cdot y$ et $x\succ y=y\cdot x$. On obtient alors que $(\mathcal{Z},\prec,\succ)$ est une algèbre dendriforme. Réciproquement, si $(\mathcal{D},\prec,\succ)$ est une algèbre dendriforme commutative (\emph{i.e.} pour tout élément $(x,y)\in\mathcal{D}^{2}$, $x\succ y=y\prec x$) alors,
	$(\mathcal{D,\prec})$ est une algèbre de Zinbiel.
\end{prop}

\subparagraph{Cas de $\dual{\SH}$.}
Introduisons d'abord les notations suivantes. Soient $n_{1},n_{2}\in\N^{*}$, on pose
\begin{align*}
Bat_{1}(n_{1},n_{2})=&\{\rho\in Bat(n_{1},n_{2})\text{, } \rho^{-1}(1)=1\},\\
Bat_{2}(n_{1},n_{2})=&\{\rho\in Bat(n_{1},n_{2})\text{, } \rho^{-1}(1)=n_{1}+1\}.  
\end{align*}
Soient $\sigma_{1}\in\Sn{1}$ et $\sigma_{2}\in\Sn{2}$. Définissons sur $\dual{\SH}$ les produits suivants:
\begin{align*}
Z_{\sigma_{1}}\nwarrow Z_{\sigma_{2}}=&\sum_{\substack{\tau\in Bat_{1}(n_{1},n_{2}), \\ \mu\in Bat_{1}(n_{1},n_{2})}}Z_{\tau\circ(\sigma_{1}\ast \sigma_{2})\circ\mu^{-1}},
&Z_{\sigma_{1}}\swarrow Z_{\sigma_{2}}=&\sum_{\substack{\tau\in Bat_{1}(n_{1},n_{2}), \\ \mu\in Bat_{2}(n_{1},n_{2})}}Z_{\tau\circ(\sigma_{1}\ast \sigma_{2})\circ\mu^{-1}},\\
Z_{\sigma_{1}}\nearrow Z_{\sigma_{2}}=&\sum_{\substack{\tau\in Bat_{2}(n_{1},n_{2}), \\ \mu\in Bat_{1}(n_{1},n_{2})}}Z_{\tau\circ(\sigma_{1}\ast \sigma_{2})\circ\mu^{-1}}, 
&Z_{\sigma_{1}}\searrow Z_{\sigma_{2}}=&\sum_{\substack{\tau\in Bat_{2}(n_{1},n_{2}), \\ \mu\in Bat_{2}(n_{1},n_{2})}}Z_{\tau\circ(\sigma_{1}\ast \sigma_{2})\circ\mu^{-1}}.
\end{align*}
\Ex{}
\begin{align*}
Z_{x_{2}x_{1}}\nwarrow Z_{x_{1}}=&Z_{x_{2}x_{1}x_{3}}+Z_{x_{2}x_{3}x_{1}}+Z_{x_{3}x_{1}x_{2}}+Z_{x_{3}x_{2}x_{1}},\\
Z_{x_{2}x_{1}}\swarrow Z_{x_{1}}=&Z_{x_{3}x_{2}x_{1}}+Z_{x_{2}x_{3}x_{1}},\\
Z_{x_{2}x_{1}}\nearrow Z_{x_{1}}=&Z_{x_{3}x_{2}x_{1}}+Z_{x_{3}x_{1}x_{2}},\\
Z_{x_{2}x_{1}}\searrow Z_{x_{1}}=&Z_{x_{1}x_{3}x_{2}}.
\end{align*}
A partir des produits précédents on construit:
\begin{align*}
Z_{\sigma_{1}}\prec Z_{\sigma_{2}}=&Z_{\sigma_{1}}\nwarrow Z_{\sigma_{2}}+Z_{\sigma_{1}}\swarrow Z_{\sigma_{2}}
&Z_{\sigma_{1}}\succ Z_{\sigma_{2}}=&Z_{\sigma_{1}}\nearrow Z_{\sigma_{2}}+Z_{\sigma_{1}}\searrow Z_{\sigma_{2}}\\
=&\sum_{\substack{\tau\in Bat_{1}(n_{1},n_{2}), \\ \mu\in Bat(n_{1},n_{2})}}Z_{\tau\circ(\sigma_{1}\ast \sigma_{2})\circ\mu^{-1}},
&=&\sum_{\substack{\tau\in Bat_{2}(n_{1},n_{2}), \\ \mu\in Bat(n_{1},n_{2})}}Z_{\tau\circ(\sigma_{1}\ast \sigma_{2})\circ\mu^{-1}}.
\end{align*}
\begin{align*}
Z_{\sigma_{1}}\wedge Z_{\sigma_{2}}=&Z_{\sigma_{1}}\nwarrow Z_{\sigma_{2}}+Z_{\sigma_{1}}\nearrow Z_{\sigma_{2}}
&Z_{\sigma_{1}}\vee Z_{\sigma_{2}}=&Z_{\sigma_{1}}\swarrow Z_{\sigma_{2}}+Z_{\sigma_{1}}\searrow Z_{\sigma_{2}}\\
=&\sum_{\substack{\tau\in Bat(n_{1},n_{2}), \\ \mu\in Bat_{1}(n_{1},n_{2})}}Z_{\tau\circ(w_{\sigma_{1}}\ast w_{\sigma_{2}})\circ\mu^{-1}}, 
&=&\sum_{\substack{\tau\in Bat(n_{1},n_{2}), \\ \mu\in Bat_{2}(n_{1},n_{2})}}Z_{\tau\circ(w_{\sigma_{1}}\ast w_{\sigma_{2}})\circ\mu^{-1}}.
\end{align*}
\Ex{}
\begin{align*}
Z_{x_{2}x_{1}}\prec Z_{x_{1}}=&Z_{x_{2}x_{1}x_{3}}+2Z_{x_{2}x_{3}x_{1}}+2Z_{x_{3}x_{2}x_{1}}+Z_{x_{3}x_{1}x_{2}},\\
Z_{x_{2}x_{1}}\succ Z_{x_{1}}=&Z_{x_{3}x_{2}x_{1}}+Z_{x_{3}x_{1}x_{2}}+Z_{x_{1}x_{3}x_{2}},\\
Z_{x_{2}x_{1}}\wedge Z_{x_{1}}=&Z_{x_{2}x_{1}x_{3}}+Z_{x_{2}x_{3}x_{1}}+2Z_{x_{3}x_{1}x_{2}}+2Z_{x_{3}x_{2}x_{1}},\\
Z_{x_{2}x_{1}}\vee Z_{x_{1}}=&Z_{x_{3}x_{2}x_{1}}+Z_{x_{2}x_{3}x_{1}}+Z_{x_{1}x_{3}x_{2}}.
\end{align*}

Par calcul direct, on obtient que $(\SH,\swarrow,\nwarrow,\nearrow,\searrow)$ est une quadri-algèbre. Il s'en suit que $(\SH,\prec,\succ)$ et $(\SH,\wedge,\vee)$ sont des algèbres dendriformes. Il est clair qu'elles sont commutatives et donc que $(\SH,\prec)$ et $(\SH,\wedge)$ sont des algèbres de Zinbiel.

En revanche, les algèbres dendriformes ne respectent pas les compatibilités décrites dans \cite{Foissy2007} vis-à-vis du coproduit; même en utilisant le coproduit co-opposé $\De^{cop}$. Pour s'en convaincre, il suffit de considérer $\Delta(Z_{x_{1}}\prec Z_{x_{1}x_{2}})$, $\Delta(Z_{x_{2}x_{1}}\prec Z_{x_{2}x_{1}x_{3}})$ et $\Delta(Z_{x_{2}x_{1}x_{3}}\wedge Z_{x_{1}})$.

\subsection{Mots tassés stricts croissants}
\subsubsection{Algèbre de Hopf des mots tassés stricts croissants $\ISPW$}
On considère le sous-espace vectoriel de $\WMat$ défini par: $$\Idl{V}=\langle w\in \WMat,~x_{0}\in \AAlph(w)\rangle.$$
Il est immédiat que $\Idl{V}$ est un biidéal de Hopf. On note $\SPW$ l'algèbre de Hopf quotient $\WMat/\Idl{V}$. On s'intéresse au sous-espace vectoriel $\langle \overline{\underbrace{x_{1}\dots x_{1}}_{\alpha_{1} \text{ fois}}\dots\underbrace{x_{n}\dots x_{n}}_{\alpha_{n} \text{ fois}}}\in\SPW,~(n,\alpha_{1},\dots,\alpha_{n})\in(\N^{*})^{n+1} \rangle$. Il s'agit d'une sous-algèbre de Hopf de $\SPW$; on la note $\ISPW$. Par commodité, on écrira toujours les éléments de $\ISPW$ sous forme de combinaison linéaire de mots tassés de la forme $\underbrace{x_{1}\dots x_{1}}_{\alpha_{1} \text{ fois}}\dots\underbrace{x_{n}\dots x_{n}}_{\alpha_{n} \text{ fois}}$ et non sous forme de combinaison linéaire de classes. L'algèbre de Hopf $\ISPW$ sera appelée algèbre de Hopf des mots tassés stricts croissants. 
\Ex{}
$$x_{1}x_{2}x_{2}x_{3}\in\ISPW,~x_{1}x_{0}x_{2}\notin\ISPW,~\overline{x_{2}x_{2}x_{1}-x_{0}x_{1}}=x_{1}x_{2}x_{2}\in\ISPW.$$

\begin{prop}
	Dans $\ISPW$ le produit s'exprime de la façon suivante:
	$$\ast:\left\{\begin{array}{rcl} \ISPW\ot\ISPW &\longrightarrow & \ISPW \\ u \ot v &\longrightarrow &
	uT_{\sup(u)}(v).\end{array}\right.$$	
	Quant au coproduit, il s'exprime de la manière suivante:
	$$\Delta:\left\{\begin{array}{rcl} \ISPW &\longrightarrow & \ISPW\otimes\ISPW \\ \underbrace{x_{1}\dots x_{1}}_{\alpha_{1} \text{ fois}}\dots\underbrace{x_{n}\dots x_{n}}_{\alpha_{n} \text{ fois}} &\longrightarrow &
	\displaystyle\sum_{\substack{0\leq p\leq k \\ I+U=\{1,\dots,k\}\\I=\{i_{1}<\dots<i_{p}\}\\U=\{u_{1}<\dots<u_{s}\}\\ s=k-p}}\underbrace{x_{1}\dots x_{1}}_{\alpha_{i_{1}} \text{ fois}}\dots\underbrace{x_{p}\dots x_{p}}_{\alpha_{i_{p}} \text{ fois}}\ot \underbrace{x_{1}\dots x_{1}}_{\alpha_{u_{1}} \text{ fois}}\dots\underbrace{x_{s}\dots x_{s}}_{\alpha_{u_{s}} \text{ fois}}.\end{array}\right.$$
\end{prop}
\Ex{}
\begin{align*}
x_{1}x_{2}x_{2}x_{3}x_{4}\ast x_{1}x_{2}x_{2}x_{2}&=x_{1}x_{2}x_{2}x_{3}x_{4}x_{5}x_{6}x_{6}x_{6},\\
\De(x_{1}x_{1}x_{1})&=x_{1}x_{1}x_{1}\ot1+ 1\ot x_{1}x_{1}x_{1},\\
\De(x_{1}x_{2}x_{2})&=x_{1}x_{2}x_{2}\ot1 +x_{1}\ot x_{1}x_{1} + x_{1}x_{1}\ot x_{1} +1\ot x_{1}x_{2}x_{2}.
\end{align*}

Il est immédiat que $\ISPW$ est librement engendrée par l'ensemble $\big\{\underbrace{x_{1}\dots x_{1}}_{n \text{ fois}}, n\in\N^{*}\big\}$. Sa cogèbre sous-jacente est trivialement cocommutative; l'algèbre de Hopf $\ISPW$ est donc isomorphe à l'algèbre enveloppante de ses éléments primitifs. On montrera dans la proposition \ref{nsymispw} que $\ISPW$ et l'algèbre de Hopf des fonctions symétriques non commutatives sont isomorphes. Un isomorphisme explicite sera donné dans la proposition \ref{isoNSymISPW} de la section \ref{morphQSym}.

On peut facilement calculer la série formelle de $\ISPW$. En effet, pour tout entier naturel $n$ non nul,
$$\bigg|\big\{w\in(\ISPW)_{n} \text{, w mot tassé}\big\}\bigg|=\bigg|\big\{\alpha\vDash n\big\}\bigg|=2^{n-1}.$$
On obtient donc : $$F_{\ISPW}(h)=1+\sum_{n=1}^{\infty}2^{n-1}h^{n}.$$
Ainsi, la série formelle des primitifs est donnée par la séquence A059966 de \cite{Sloane} : 
$$F_{\Prim{\ISPW}{}}(h)=h+h^{2}+2h^{3}+3h^{4}+6h^{5}+9h^{6}+18h^{7}+30h^{8}+56h^{9}+99h^{10}+\dots$$

L'objectif de la suite de cette sous-section est alors de déterminer quelques familles d'éléments primitifs de $\ISPW$.

\begin{prop}\label{formebaseparticuliere} Soit $n\in\N^{*}$. On définit $d_{n}$ comme étant la dimension de $\Prim{\ISPW}{n}$. Il existe une base $\B=(e_{1},\dots,e_{d_{n}})$ de $\Prim{\ISPW}{n}$ telle que chaque élément de $\B$ soit associé à une seule partition de $n$. Plus précisément, pour tout entier $i\in\{1,\dots,d_{n}\}$, il existe une unique partition $(\alpha_{1}\leq\dots\leq\alpha_{k_{i}})$  de $n$ telle que $$e_{i}\in Vect\Bigg(\underbrace{x_{1}\dots x_{1}}_{\alpha_{\sigma(1)} \text{ fois}}\dots\underbrace{x_{k_{i}}\dots x_{k_{i}}}_{\alpha_{\sigma(k_{i})} \text{ fois}}, \sigma\in\Sn{k_{i}}\Bigg).$$
\end{prop}

\begin{proof}
	On sait que $\Prim{\ISPW}{}$, muni du crochet $\{-,-\}$ associé au produit de concaténation décalé, est l'algèbre de Lie librement engendrée par l'ensemble $\bigg\{\underbrace{x_{1}\dots x_{1}}_{n \text{ fois}}, n\in\N^{*}\bigg\}$.
\end{proof}

\Ex{1} Considérons les mots tassés stricts croissants de degré 3. On obtient facilement que $\Prim{\ISPW}{3}$ est un espace vectoriel de dimension 2 engendré par les vecteurs $x_{1}x_{1}x_{1}$ et $x_{1}x_{1}x_{2}-x_{1}x_{2}x_{2}$. Le nombre 3 possède trois partitions ((3), (1,2) et (1,1,1)) et quatre compositions ((3), (1,2), (2,1) et (1,1,1)). Le mot $x_{1}x_{1}x_{1}$ est formé d'un bloc de trois lettres identiques. Il est donc associé à la partition (3). La somme $x_{1}x_{1}x_{2}-x_{1}x_{2}x_{2}$ est la combinaison linéaire d'un mot dont la taille des blocs de lettres distinctes est régie par la composition (2,1) et d'un mot régi par la composition (1,2). On associe donc l'élément à la partition (1,2). Les partitions du type $(\underbrace{1,\dots,1}_{n\geq2 \text{ fois}})$ ne sont associées à aucun élément primitif non nul.

\begin{prop}\label{prim}
	Soient $n\in\N^{*}$, $(\alpha_{1},\dots,\alpha_{n})\in(\N^{*})^{n}$ et $(\beta_{1},\dots,\beta_{n})\in(\N^{*})^{n}$ tel que pour $i\neq j$ on ait $\beta_{i}\neq\beta_{j}$. Posons $\rho=\alpha_{2}+\dots+\alpha_{n}$ et $\theta=\alpha_{1}+\dots+\alpha_{n}$. Appelons $\gamma$ le $\theta$-uplet $\gamma=(\gamma_{1},\dots,\gamma_{\theta})=(\underbrace{\beta_{1},\dots,\beta_{1}}_{\alpha_{1} \text{ fois}},\dots,\underbrace{\beta_{n},\dots,\beta_{n}}_{\alpha_{n} \text{ fois}})$. Dans $\ISPW$, on définit l'élément
	$$\Pn{\gamma}=\frac{1}{\alpha_{2}!\dots\alpha_{n}!}\sum_{\sigma\in\Sn{\rho}}\sum_{k=1}^{\alpha_{1}}(-1)^{k-1}\binom{\alpha_{1}-1}{k-1}\sum_{s=0}^{\rho}(-1)^{s}\binom{\rho}{s}w_{\gamma,\tilde{\sigma}_{k,s}}$$
	où
	\begin{align*}
	\tilde{\sigma}_{1,0}=&\begin{pmatrix}1 & \dots & \alpha_{1} & \alpha_{1}+1 & \dots & \theta \\1 & \dots & \alpha_{1} & \sigma(1)+\alpha
	_{1} & \dots & \sigma(\rho)+\alpha_{1}\end{pmatrix}\in\Sn{\theta},\\
	\tilde{\sigma}_{k,0}=&\begin{pmatrix}1 & \dots & \alpha_{1}-k+1 & \alpha_{1}-k+2 & \dots & \theta-k+1 & \theta-k+2 & \dots & \theta \\1 & \dots & \alpha_{1}-k+1 & \sigma(1)+\alpha_{1} & \dots & \sigma(\rho)+\alpha_{1} & \alpha_{1}-k+2 & \dots & \alpha_{1}\end{pmatrix}\in\Sn{\theta},\\
	\tau_{k,s}=& \begin{pmatrix}\alpha_{1}-k+s & \alpha_{1}-k+s+1 \end{pmatrix}\in\Sn{\theta} \text{ pour } 1\leq k \leq \alpha_{1} \text{ et } 1\leq s \leq \rho,\\
	\tilde{\sigma}_{k,s}=&\tilde{\sigma}_{k,s-1}\circ\tau_{k,s}\in\Sn{\theta} \text{ pour } 1\leq k \leq \alpha_{1} \text{ et } 1\leq s \leq \rho,\\
	w_{\gamma,\tilde{\sigma}_{k,s}}=&\underbrace{x_{1}\dots x_{1}}_{\gamma_{\tilde{\sigma}_{k,s}(1)} \text{ fois}}\dots\underbrace{x_{\theta}\dots x_{\theta}}_{\gamma_{\tilde{\sigma}_{k,s}(\theta)} \text{ fois}}.
	\end{align*}
	$\Pn{\gamma}$ est un élément primitif.
\end{prop}

\begin{proof}
	Soit $n\in\N^{*}$ et notons $\{-,-\}$ le crochet de Lie associé au produit de $\ISPW$. Commençons par montrer que $\Pn{(1,\dots,n)}$ est primitif. Ceci est vrai car $$\Pn{(1,\dots,n)}=\{\{\dots\{x_{1},x_{1}x_{1}\}\dots\},\underbrace{x_{1}\dots x_{1}}_{n\text{ fois}}\}.$$
	
	Soient maintenant $(\alpha_{1},\dots,\alpha_{n})\in(\N^{*})^{n}$ et $(\beta_{1},\dots,\beta_{n})\in(\N^{*})^{n}$ tel que pour $i\neq j$ on ait $\beta_{i}\neq\beta_{j}$. 
	Définissons les entiers $\rho=\alpha_{2}+\dots+\alpha_{n}$. Pour tout entier $i$ de l'intervalle $\llbracket1,\alpha_{1}\rrbracket$ on pose $\theta_{i}=i+\alpha_{2}+\dots+\alpha_{n}$ et on appelle $\gamma_{i}$ le $\theta_{i}$-uplet $\gamma_{i}=(\underbrace{\beta_{1},\dots,\beta_{1}}_{i \text{ fois}},\underbrace{\beta_{2},\dots,\beta_{2}}_{\alpha_{2} \text{ fois}},\dots,\underbrace{\beta_{n},\dots,\beta_{n}}_{\alpha_{n} \text{ fois}})$. 
	
	On considère l'endomorphisme d'algèbres de Hopf 
	$$\Lambda_{\beta}:\left\{\begin{array}{rcl}\ISPW &\longrightarrow &\ISPW \\ \underbrace{x_{1}\dots x_{1}}_{k \text{ fois}} &\longrightarrow & \begin{cases} \underbrace{x_{1}\dots x_{1}}_{\beta_{k} \text{ fois}} & \mbox{ si } 1\leq k\leq n,\\ \underbrace{x_{1}\dots x_{1}}_{k \text{ fois}} & \mbox{ sinon, }\end{cases} \end{array}\right.$$	
	
	Il est immédiat que $\Pn{\gamma_{1}}=\frac{1}{\alpha_{2}!\dots\alpha_{n}!}\Lambda_{\beta}(\Pn{(1,\dots,\theta_{1})})$. Pour tout entier $i$ de  l'intervalle $\llbracket1,\alpha_{1}-1\rrbracket$ on a
	$$\Pn{\gamma_{i+1}}=\{\underbrace{x_{1}\dots x_{1}}_{\beta_{1} \text{ fois}},\Pn{\gamma_{i}} \}.$$
\end{proof}	
\Ex{}
\begin{align*}
\Pn{(1,1,2,2)}=&x_{1}x_{2}x_{3}x_{3}x_{4}x_{4}-2x_{1}x_{2}x_{2}x_{3}x_{4}x_{4}+2x_{1}x_{1}x_{2}x_{3}x_{3}x_{4}-x_{1}x_{1}x_{2}x_{2}x_{3}x_{4},\\
\Pn{(2,2,1,3)}=&x_{1}x_{1}x_{2}x_{2}x_{3}x_{4}x_{4}x_{4}-2x_{1}x_{1}x_{2}x_{3}x_{3}x_{4}x_{4}x_{4}+2x_{1}x_{2}x_{2}x_{3}x_{3}x_{3}x_{4}x_{4}\\-&x_{1}x_{2}x_{2}x_{2}x_{3}x_{3}x_{4}x_{4}+x_{1}x_{1}x_{2}x_{2}x_{3}x_{3}x_{3}x_{4}-2x_{1}x_{1}x_{2}x_{2}x_{2}x_{3}x_{3}x_{4}\\
+&2x_{1}x_{1}x_{1}x_{2}x_{2}x_{3}x_{4}x_{4}+x_{1}x_{1}x_{1}x_{2}x_{3}x_{3}x_{4}x_{4}.
\end{align*}
\Rq{1}	En degré $n\leq6$, la proposition \ref{prim} permet d'obtenir une base $\Prim{\ISPW}{n}$. En degré 7, cette proposition ne donne que 17 vecteurs linéairement indépendants sur les 18 nécessaires pour former une base. 

\begin{prop}\label{prim_Lambda}
	Soient $n\in\N^{*}$, $(\alpha_{1},\dots,\alpha_{n})\in(\N^{*})^{n}$ et $(\beta_{1},\dots,\beta_{n})\in(\N^{*})^{n}$ tel que pour $i\neq j$ on ait $\beta_{i}\neq\beta_{j}$. Posons $\theta=\alpha_{1}+\dots+\alpha_{n}$ et $\gamma=(\gamma_{1},\dots,\gamma_{\theta})=(\underbrace{\beta_{1},\dots,\beta_{1}}_{\alpha_{1} \text{ fois}},\dots,\underbrace{\beta_{n},\dots,\beta_{n}}_{\alpha_{n} \text{ fois}})$. Dans $\ISPW$, on définit l'élément
	$$\Pn{\Lambda_{\gamma}}=\cfrac{1}{\alpha_{1}!\dots\alpha_{n}!}\sum_{\sigma\in\Sn{\theta}}\sum_{i=1}^{\alpha_{1}}(-1)^{\sigma^{-1}(i)-1}\binom{\theta-1}{\sigma^{-1}(i)-1}w_{\gamma,\sigma}$$
	où $$w_{\gamma,\sigma}=\underbrace{x_{1}\dots x_{1}}_{\gamma_{\sigma(1)} \text{ fois}}\dots\underbrace{x_{\theta}\dots x_{\theta}}_{\gamma_{\sigma(\theta)} \text{ fois}}.$$
	L'élément $\Pn{\Lambda_{\gamma}}$ est primitif dans $\ISPW$.
\end{prop}

\begin{proof}
	Soient $n\in\N^{*}$, $(\alpha_{1},\dots,\alpha_{n})\in(\N^{*})^{n}$ et $(\beta_{1},\dots,\beta_{n})\in(\N^{*})^{n}$ tel que pour $i\neq j$ on ait $\beta_{i}\neq\beta_{j}$. Posons $\theta=\alpha_{1}+\dots+\alpha_{n}$ et $\gamma=(\gamma_{1},\dots,\gamma_{\theta})=(\underbrace{\beta_{1},\dots,\beta_{1}}_{\alpha_{1} \text{ fois}},\dots,\underbrace{\beta_{n},\dots,\beta_{n}}_{\alpha_{n} \text{ fois}})$. 
	Considérons l'élément $\Pn{\Lambda_{(1,\dots,\theta)}}$.
	on a: $$\Pn{\Lambda_{(1,\dots,\theta)}}=\sum_{\sigma\in\Sn{\theta}}(-1)^{\sigma^{-1}(1)-1}\binom{\theta-1}{\sigma^{-1}(1)-1}\underbrace{x_{1}\dots x_{1}}_{\sigma(1) \text{ fois}}\dots\underbrace{x_{\theta}\dots x_{\theta}}_{\sigma(\theta) \text{ fois}}=\Pn{(1,\dots,\theta)}.$$
	Ainsi $\Pn{\Lambda_{(1,\dots,\theta)}}$ est un élément primitif de $\ISPW$.

	Grâce à l'application $\Lambda_{\gamma}$ on a: $\Pn{\Lambda_{\gamma}}=\cfrac{1}{(\alpha_{1}-1)!\alpha_{2}!\dots\alpha_{n}!}\Lambda_{\gamma}(\Pn{\Lambda_{(1,\dots,\theta)}})$. L'élément $\Pn{\Lambda_{\gamma}}$ est donc primitif dans $\ISPW$.
\end{proof}

\Ex{1} Posons $\gamma=(1,1,3,3,3)$. L'élément $\Pn{\Lambda_{\gamma}}$ est primitif dans $\ISPW$. 
\begin{align*}
	\Pn{\Lambda_{\gamma}}=&-3x_{1}x_{2}x_{3}x_{3}x_{3}x_{4}x_{4}x_{4}x_{5}x_{5}x_{5}+7x_{1}x_{2}x_{2}x_{2}x_{3}x_{4}x_{4}x_{4}x_{5}x_{5}x_{5}+2x_{1}x_{1}x_{1}x_{2}x_{3}x_{4}x_{4}x_{4}x_{5}x_{5}x_{5}\\
	-&3x_{1}x_{2}x_{2}x_{2}x_{3}x_{3}x_{3}x_{4}x_{5}x_{5}x_{5}-8x_{1}x_{1}x_{1}x_{2}x_{3}x_{3}x_{3}x_{4}x_{5}x_{5}x_{5}+2x_{1}x_{1}x_{1}x_{2}x_{2}x_{2}x_{3}x_{4}x_{5}x_{5}x_{5}\\
	+&2x_{1}x_{2}x_{2}x_{2}x_{3}x_{3}x_{3}x_{4}x_{4}x_{4}x_{5}-3x_{1}x_{1}x_{1}x_{2}x_{3}x_{3}x_{3}x_{4}x_{4}x_{4}x_{5}+7x_{1}x_{1}x_{1}x_{2}x_{2}x_{2}x_{3}x_{4}x_{4}x_{4}x_{5}\\
	-&3x_{1}x_{1}x_{1}x_{2}x_{2}x_{2}x_{3}x_{3}x_{3}x_{4}x_{5}.
\end{align*}

\Rq{} \begin{enumerate}
	\item Soient $n\in\N^{*}$, $(\alpha_{1},\dots,\alpha_{n})\in(\N^{*})^{n}$ et $(\beta_{1},\dots,\beta_{n})\in(\N^{*})^{n}$ tel que pour $i\neq j$ on ait $\beta_{i}\neq\beta_{j}$. Posons $\gamma=(\underbrace{\beta_{1},\dots,\beta_{1}}_{\alpha_{1} \text{ fois}},\dots,\underbrace{\beta_{n},\dots,\beta_{n}}_{\alpha_{n} \text{ fois}})$.
	\begin{enumerate}
		\item Si $\alpha_{1}=1$, les éléments $\Pn{\Lambda_{\gamma}}$ et $\Pn{\gamma}$ sont égaux.
		\item Si ($n=2$ et $\alpha_{1}=\alpha_{2}=2$) ou ($n=3$, $\alpha_{1}=2$ et $\alpha_{2}=\alpha_{3}=1$), \emph{ i.e.} $\gamma=(\beta_{1},\beta_{1},\beta_{2},\beta_{2})$ ou $\gamma=(\beta_{1},\beta_{1},\beta_{2},\beta_{3})$, on a alors: $\Pn{\Lambda_{\gamma}}=-2\Pn{\gamma}$.
		\item Si $n=2$ et $\alpha_{2}=1$, \emph{ i.e.} $\gamma=(\underbrace{\beta_{1},\dots,\beta_{1}}_{\alpha_{1} \text{ fois}},\beta_{2})$, on a alors $\Pn{\Lambda_{\gamma}}=(-1)^{\alpha_{1}+1}\Pn{\gamma}$.
		\item Si ($n=2$, $\alpha_{1}\geq2$ et $\alpha_{2}\geq3$) ou ($n=3$, $\alpha_{1}\geq2$ et $\alpha_{2}+\alpha_{3}\neq2$) ou ($n\geq4$ et $\alpha_{1}\geq2$), il n'existe pas de scalaire $k\in\K^{*}$ tel que $\Pn{\Lambda_{\gamma}}=k\Pn{\gamma}$. En effet, en posant $\alpha_{3}=0$ dans le cas $n=2$ la multiplicité du mot tassé 
		$$\underbrace{x_{1}\dots x_{1}}_{\beta_{2} \text{ fois}}\ast(\underbrace{x_{1}\dots x_{1}}_{\beta_{1} \text{ fois}})^{\alpha_{1}}\ast(\underbrace{x_{1}\dots x_{1}}_{\beta_{2} \text{ fois}})^{\alpha_{2}-1}\ast(\underbrace{x_{1}\dots x_{1}}_{\beta_{3} \text{ fois}})^{\alpha_{3}}\ast\dots\ast(\underbrace{x_{1}\dots x_{1}}_{\beta_{n} \text{ fois}})^{\alpha_{n}}$$
		est nulle dans $\Pn{\gamma}$ alors quelle est non nulle, puisqu'égale à $\displaystyle(-1)^{\alpha_{1}}\binom{\theta-2}{\alpha_{1}}-1$, dans $\Pn{\Lambda_{\gamma}}$.
		\item Si ($n=2$, $\alpha_{1}\geq3$ et $\alpha_{2}=2$) ou ($n=3$, $\alpha_{1}\geq3$ et $\alpha_{2}=\alpha_{3}=1$), il n'existe pas de scalaire $k\in\K^{*}$ tel que $\Pn{\Lambda_{\gamma}}=k\Pn{\gamma}$. En effet, en posant $\alpha_{3}=0$ dans le cas $n=2$, la multiplicité du mot tassé 
		$$\underbrace{x_{1}\dots x_{1}}_{\beta_{1} \text{ fois}}\ast\underbrace{x_{1}\dots x_{1}}_{\beta_{2} \text{ fois}}\ast(\underbrace{x_{1}\dots x_{1}}_{\beta_{1} \text{ fois}})^{\alpha_{1}-1}\ast(\underbrace{x_{1}\dots x_{1}}_{\beta_{1} \text{ fois}})^{\alpha_{2}-1}\ast(\underbrace{x_{1}\dots x_{1}}_{\beta_{3} \text{ fois}})^{\alpha_{3}}$$
		est nulle dans $\Pn{\gamma}$ alors quelle est non nulle, puisqu'égale à $\theta-1+(-1)^{\theta}$, dans $\Pn{\Lambda_{\gamma}}$.
	\end{enumerate}
	\item La famille définie dans la proposition \ref{prim_Lambda} permet d'engendrer $\Prim{\ISPW}{n}$ pour $n$ inférieur ou égal à $6$ mais ne permet pas d'engendrer $\Prim{\ISPW}{7}$. Il en est de même si l'on considère la réunion des familles définies dans les propositions \ref{prim} et \ref{prim_Lambda}.
\end{enumerate}

\subsubsection{Algèbres de Hopf $\dual{\ISPW}$}
L'objectif ici est de présenter les opérations de l'algèbre de Hopf $\dual{\ISPW}$.
Pour cela, notons $(Z_{\underbrace{x_{1}\dots x_{1}}_{\alpha_{1} \text{ fois}}\dots\underbrace{x_{n}\dots x_{n}}_{\alpha_{n} \text{ fois}}})_{(n,\alpha_{1},\dots,\alpha_{n})\in(\N^{*})^{n+1}}$ la base duale des mots tassés stricts croissants.
\begin{prop}
	Le coproduit de la cogèbre $\dual{\ISPW}$ est donné par l'application suivante:
	$$\De:\left\{\begin{array}{rcl} \dual{\ISPW} &\longrightarrow & \dual{\ISPW}\ot\dual{\ISPW}  \\ 
	Z_{\underbrace{x_{1}\dots x_{1}}_{\alpha_{1} \text{ fois}}\dots\underbrace{x_{n}\dots x_{n}}_{\alpha_{n} \text{ fois}}} &\longrightarrow & \sum\limits_{i=0}^{n}Z_{\underbrace{x_{1}\dots x_{1}}_{\alpha_{1} \text{ fois}}\dots\underbrace{x_{i}\dots x_{i}}_{\alpha_{i} \text{ fois}}}\otimes Z_{\underbrace{x_{1}\dots x_{1}}_{\alpha_{i+1} \text{ fois}}\dots\underbrace{x_{n-i}\dots x_{n-i}}_{\alpha_{n} \text{ fois}}}. \end{array}\right.$$	
\end{prop}

\Ex{1}
$$\De(Z_{x_{1}x_{1}x_{2}x_{3}x_{3}x_{3}})=Z_{x_{1}x_{1}x_{2}x_{3}x_{3}x_{3}}\ot1 +Z_{x_{1}x_{1}}\ot Z_{x_{1}x_{2}x_{2}x_{2}}+Z_{x_{1}x_{1}x_{2}}\ot Z_{x_{1}x_{1}x_{1}}+1\ot Z_{x_{1}x_{1}x_{2}x_{3}x_{3}x_{3}}.$$

\begin{prop}
	Le produit de l'algèbre $\dual{\ISPW}$ est défini par l'application: 
	$$m:\left\{\begin{array}{rcl} \dual{\ISPW}\ot\dual{\ISPW} &\longrightarrow & \dual{\ISPW}  \\ 
	Z_{\underbrace{x_{1}\dots x_{1}}_{\alpha_{1} \text{ fois}}\dots\underbrace{x_{p}\dots x_{p}}_{\alpha_{p} \text{ fois}}}\ot Z_{\underbrace{x_{1}\dots x_{1}}_{\beta_{1} \text{ fois}}\dots\underbrace{x_{q}\dots x_{q}}_{\beta_{q} \text{ fois}}} &\longrightarrow & \sum\limits_{\tau\in Bat(p,q)}Z_{\underbrace{x_{1}\dots x_{1}}_{\gamma_{\tau^{-1}(1)} \text{ fois}}\dots\underbrace{x_{p+q}\dots x_{p+q}}_{\gamma_{\tau^{-1}(p+q)} \text{ fois}}} \end{array}\right.$$
	où $$\gamma_{i}=\begin{cases}
	\alpha_{i} & \mbox{ si } i\in\{1,\dots,p\}, \\
	\beta_{i-p} & \mbox{ si } i\in\{p+1,\dots,p+q\}. 
	\end{cases}$$
\end{prop}
\begin{proof}
	Ceci se montre par calcul direct.
\end{proof}
\Ex{}
\begin{align*}
Z_{x_{1}x_{1}}Z_{x_{1}x_{1}}=&2Z_{x_{1}x_{1}x_{2}x_{2}},\\
Z_{x_{1}x_{2}x_{2}}Z_{x_{1}x_{1}}=&2Z_{x_{1}x_{2}x_{2}x_{3}x_{3}}+Z_{x_{1}x_{1}x_{2}x_{3}x_{3}}.
\end{align*}

\paragraph{Isomorphisme entre $\dual{\ISPW}$ et $\QSym$.}\label{QSymdef} L'algèbre des fonctions quasi-symétriques, notée $\QSym$, a été introduite  par Gessel \cite{Gessel1984} suite au développement de la théorie des P-partitions par Stanley \cite{Stanley1972}. Malvenuto et Reutenauer \cite[chapitre 4]{Malvenuto1994} \cite{Malvenuto1995} étudient la structure d'algèbre de Hopf de $\QSym$. Ils explicitent l'antipode \cite[chapitre 4, corollaire 4.20]{Malvenuto1994} \cite[corollaire 2.3]{Malvenuto1995} également déterminée par Ehrenborg \cite[section 3, proposition 3.4]{Ehrenborg1996}. L'algèbre de Hopf des fonctions symétriques non commutatives $\NSym$ est introduite dans \cite[section 3.1]{Gelfand1995} et étudiée \cite{Gelfand1995,Krob1997,Duchamp1997,Krob1997a,Krob1999,Duchamp2002,Duchamp2011}. On peut la voir comme le dual gradué de $\QSym$. Malvenuto et Reutenauer s'intéressent également à $\dual{\QSym}$.  En effet, les auteurs définissent les fonctions quasi-symétriques libres \cite[chapitre 5, section 5.2]{Malvenuto1994} \cite[section 3]{Malvenuto1995}, les munissent de deux structures d'algèbres de Hopf isomorphes et duales, expriment l'algèbre $\Sigma$ des descentes de Solomon \cite{Solomon1976,Garsia1989} comme sous-algèbre de Hopf et montrent que $\Sigma$ et le dual gradué de $\QSym$ sont isomorphes \cite[chapitre 5, théorème 5.18]{Malvenuto1994} \cite[théorème 3.3]{Malvenuto1995}. 

Rappelons ici la construction de $\QSym$ détaillée dans \cite{Malvenuto1994,Malvenuto1995}.  On considère $Y=\{y_{1}<y_{2}<\dots\}$ un ensemble infini, dénombrable et totalement ordonné d'indéterminées commutant deux à deux. On pose $\K[[Y]]$ l'algèbre des séries formelles sur $Y$. On considère dans $\K[[Y]]$ le sous-espace vectoriel $$\QSym=Vect\bigg(M_{(\alpha_{1},\dots,\alpha_{k})}=\sum_{y_{i_{1}}<\dots<y_{i_{k}}}y_{i_{1}}^{\alpha_{1}}\dots y_{i_{k}}^{\alpha_{k}}, k\in\N^{*},(\alpha_{1},\dots,\alpha_{k})\in(\N^{*})^{k}\bigg).$$
L'espace $\QSym$ est une sous-algèbre de $\K[[Y]]$ pour le produit habituel des séries formelles. Dans le cas de $\QSym$ il s'écrit comme un produit de battage contractant.
\Ex{1}
Le coproduit de déconcaténation $\De$ fait de $\QSym$ une cogèbre colibre compatible avec la structure d'algèbre. Ainsi, $\QSym$ est une algèbre de Hopf.
\Ex{}
\begin{align*}
\De(M_{(1)})=&M_{(1)}\ot1+ 1\ot M_{(1)},\\
\De(M_{(2,1,3)})=&M_{(2,1,3)}\ot1 +M_{(2)}\ot M_{(1,3)} +M_{(2,1)}\ot M_{(3)} +1\ot M_{(2,1,3)}.
\end{align*} 

Réécrivons maintenant les opérations de $\dual{\ISPW}$ en terme de compositions. Pour tout entier $n\in\N^{*}$ et tout $n$-uplet $(k_{1},\dots,k_{n})\in(\N^{*})^{n}$, posons $Z_{(k_{1},\dots,k_{n})}=Z_{\underbrace{x_{1}\dots x_{1}}_{k_{1} \text{ fois}}\dots\underbrace{x_{n}\dots x_{n}}_{k_{n}}}$. Le produit et le coproduit de $\dual{\ISPW}$ se réécrivent de la façon suivante:
pour tous les entiers $n,m\in\N^{*}$, tout $n$-uplet $(k_{1},\dots,k_{n})\in(\N^{*})^{n}$ et tout $m$-uplet $(l_{1},\dots,l_{m})\in(\N^{*})^{m}$,
\begin{align*}
	Z_{(k_{1},\dots,k_{n})}Z_{(l_{1},\dots,l_{m})}=&Z_{(k_{1},\dots,k_{n})\shuffle (l_{1},\dots,l_{m})},\\
	\De(Z_{(k_{1},\dots,k_{n})})=&Z_{(k_{1},\dots,k_{n})}\ot 1 + 1\ot Z_{(k_{1},\dots,k_{n})} +\sum_{i=1}^{n-1}Z_{(k_{1},\dots,k_{i})}\ot Z_{(k_{i+1},\dots,k_{n})}.
\end{align*}

\begin{prop}\label{qsymispwdual}
	Les algèbres de Hopf $\QSym$ et $\dual{\ISPW}$ sont isomorphes.
\end{prop}
\begin{proof}
	On sait que les algèbres de Hopf $\QSym$ et $\dual{\ISPW}$ ont même série de \PH{} et qu'elles sont graduées, connexes, colibres et commutatives. Elles sont donc isomorphes.
\end{proof}

\Rq{1}  Un isomorphisme explicite entre $\dual{\ISPW}$ et $\QSym$ sera donné dans la section \ref{morphQSym} par la proposition \ref{isoISPWdQSym}.\\

Rappelons maintenant les opérations de l'algèbre de Hopf des fonctions symétriques non-commutatives $\NSym$ obtenues par dualité avec celles de $\QSym$. Notons $(M_{\alpha\models n}^{*})_{n\in\N^{*}}$ la base duale de la base $(M_{\alpha\models n})_{n\in\N^{*}}$ de $\QSym$.
Le produit de $\NSym$ est donné par la concaténation des compositions.
\Ex{1}
$$M_{(1,3,2,2,1)}^{*}M_{(4,1,4)}^{*}=M_{(1,3,2,2,1,4,1,4)}^{*}.$$
Pour déterminer le coproduit de $\NSym$, il suffit de connaître sa valeur pour les éléments de la famille  $(M_{n}^{*})_{n\in\N^{*}}$. Soit donc $n$ un entier naturel non nul. La valeur de $\De(M_{n}^{*}))$ est donnée par: 
$$\De(M_{(n)}^{*})=M_{(n)}^{*}\ot 1 + 1\ot M_{(n)}^{*} +\sum_{s=1}^{n-1}M_{(s)}^{*}\ot M_{(n-s)}^{*}.$$

\Ex{} Considérons pour exemples les éléments $M_{(3)}^{*}$ et $M_{(1,2)}^{*}$. Calculons leur coproduit réduit.
\begin{align*}
\tilde{\De}(M_{(3)}^{*})=&M_{(1)}^{*}\ot M_{(2)}^{*}+M_{(2)}^{*}\ot M_{(1)}^{*},\\
\tilde{\De}(M_{(1,2)}^{*})=&M_{(1)}^{*}\ot M_{(2)}^{*} + M_{(2)}^{*}\ot M_{(1)}^{*} + M_{(1,1)}^{*}\ot M_{(1)}^{*}+  M_{(1)}^{*}\ot M_{(1,1)}^{*}.
\end{align*}
\begin{prop}\label{nsymispw}
	Les algèbres de Hopf $\NSym$ et $\ISPW$ sont isomorphes.
\end{prop} 
\begin{proof}
	Il suffit de dualiser la proposition précédente.
\end{proof}

\Rq{1}  Un isomorphisme explicite entre $\NSym$ et $\ISPW$ sera donné dans la proposition \ref{isoNSymISPW} de la section  \ref{morphQSym}.

\subsection{Algèbre de Hopf des compositions étendues}

\subsubsection{Algèbre de Hopf $\Cp$}
\paragraph{Construction de $\Cp$.}
On appelle espace vectoriel des compositions étendues, l'espace $\Cp$ défini par:
$$\Cp= Vect\Bigg((\alpha_{0},\alpha_{1},\dots,\alpha_{k})\in\N\times(\N^{*})^{k} ,k\in\N\Bigg).$$ On définit l'application linéaire surjective suivante:
$$\Pi:\left\{\begin{array}{rcl}\WMat &\longrightarrow & \Cp \\ w=x_{i_{1}}\dots x_{i_{s}} &\longrightarrow & (|w|_{x_{0}},\dots,|w|_{x_{\sup(w)}}). \end{array}\right. $$
 Pour $w=x_{i_{1}}\dots x_{i_{s}}\in \WMat$ et $\sigma\in\Sn{s}$ on pose $w_{\sigma}= x_{i_{\sigma(1)}}\dots x_{i_{\sigma(s)}}$.
Le noyau de l'application $\Pi$ est l'espace vectoriel $\Idl{J}= \ker(\Pi)=\langle w-w_{\sigma}, w\in\WMat,\sigma\in\Sn{|w|}\rangle$. 

\begin{prop}
	$\Idl{J}$ est un biidéal de Hopf de $\WMat$.
\end{prop}
\begin{proof}
	Soit $v,w\in\WMat$ deux mots tassés, $\sigma\in\Sn{|w|}$. On appelle $\alpha$, $\beta$, $k$ et $s$ les entiers respectifs $\sup(v)$, $\sup(w)$, $|v|$ et $|w|$. On a :
	$$ v\ast (w-w_{\sigma})= v\ast w - v\ast w_{\sigma}=v T_{\alpha}(w)-v T_{\alpha}(w_{\sigma}) = u-u_{\tilde{\sigma}}$$
	où $$u=v T_{\alpha}(w) \text{ et } \tilde{\sigma}=\begin{pmatrix}1 & \dots & k & k+1 & \dots & k+s \\ 1 & \dots & k & k+\sigma(1) & \dots & k+\sigma(s) \end{pmatrix}$$ 
	et $$ (w-w_{\sigma})\ast v= u'-u'_{\tilde{\sigma'}}$$ 
	où 
	$$u'=w T_{\beta}(v) \text{ et } \tilde{\sigma'}=\begin{pmatrix}1 & \dots & s & s+1 & \dots & s+k \\ \sigma(1) & \dots & \sigma(s) & s+1 & \dots & s+k \end{pmatrix}.$$
	On obtient donc que $\Idl{J}$ est un idéal de $\WMat$. 

	Considérons maintenant $\De(w-w_{\sigma})$.
	\begin{align*}
	\De&(w-w_{\sigma})=\De(w)-\De(w_{\sigma})\\
	=&\sum_{I+J=\{1,\dots,|w|\}}\bigg(\pack(w[I])\otimes \pack(w[J]/w[I])-\pack(w_{\sigma}[I])\otimes \pack(w_{\sigma}[J]/w[I])\bigg)\\
	=&\sum_{I+J=\{1,\dots,|w|\}}\bigg(\pack(w[I])\otimes \pack(w[J]/w[I])-\pack(w_{\sigma}[U_{I}])\otimes \pack(w_{\sigma}[V_{J}]/w[U_{I}])\bigg)\\
	\end{align*}
	où $U_{I}=\sigma^{-1}(I)$ et $V_{J}=\sigma^{-1}(J)$.
	On a donc $\De(w-w_{\sigma})\subset \Idl{J}\ot\WMat + \WMat\ot\Idl{J}$. Ainsi, $\Idl{J}$ est un coïdéal.
	
	Au final, $\Idl{J}$ est un biidéal de Hopf de $\WMat$.
\end{proof}

 $\WMat/\Idl{J}$ est donc une bigèbre graduée, c'est donc une algèbre de Hopf quotient de $\WMat$. En identifiant maintenant $\WMat/\Idl{J}$ et $\Cp$, on peut munir $\Cp$ d'une structure d'algèbre de Hopf. Reste alors à décrire ses opérations.
 
 \begin{prop}
 	Le produit de l'algèbre $\Cp$ est donné par l'application:
 	$$\ast:\left\{\begin{array}{rcl} \Cp\ot\Cp &\longrightarrow & \Cp \\ (\alpha_{0},\dots,\alpha_{s})\ot(\beta_{0},\dots,\beta_{k}) &\longrightarrow &(\alpha_{0}+\beta_{0},\alpha_{1},\dots,\alpha_{s},\beta_{1},\dots,\beta_{k}).\end{array}\right.$$	
 \end{prop} 
 \Ex{}
 \begin{align*}
	 (0,1,4,2)\ast(3,2,2)=&(3,1,4,2,2,2),\\
	 (3,2,2)\ast(0,1,4,2)=&(3,2,2,1,4,2),\\
	 (2,3,4,1,2)\ast(12,3,14,4)=&(14,3,4,1,2,3,14,4).
 \end{align*}
 \Rq{1} L'algèbre $\Cp$ ne peut être libre du fait de l'existence de la lettre $x_{0}$ dans l'alphabet de $\WMat$. Les éléments $(\alpha_{0})$ pour $\alpha_{0}\in\N$ sont centraux. Si on considère la sous-algèbre $\mathcal{C}$ engendrée par les éléments de la forme $(0,\alpha_{1},\dots,\alpha_{k})\in(\N^{*})^{k}$ avec $k\in\N$, celle-ci est alors libre: il s'agit de l'algèbre tensorielle $T\langle (0,n),n\in\N^{*} \rangle$.
 
 \begin{prop}
 	Le coproduit de la cogèbre $\Cp$ est donné par l'application:
 	$$\De:\left\{\begin{array}{rcl} \Cp &\longrightarrow & \Cp\ot\Cp \\ (\alpha_{0},\dots,\alpha_{p}) &\longrightarrow &\sum\limits_{a=0}^{\alpha_{0}}\sum\limits_{n=0}^{p}\sum\limits_{1\leq i_{i}<\dots<i_{n}\leq p}\sum\limits_{s=n}^{\alpha_{i_{1}}+\dots+\alpha_{i_{n}}}\sum\limits_{\substack{k_{i_{1}}+\dots+k_{i_{n}}=s\\1\leq k_{i_{j}}\leq \alpha_{i_{j}}}}c_{\alpha,a,n,I,s,K_{I}}t_{\alpha,a,n,I,s,K_{I}}\end{array}\right.$$	
 	où 
 	\begin{align*}
 	c_{\alpha,a,n,I,s,K_{I}}=&\binom{\alpha_{0}}{a}\binom{\alpha_{i_{1}}}{k_{i_{1}}}\dots\binom{\alpha_{i_{n}}}{k_{i_{n}}},\\
 	t_{\alpha,a,n,I,s,K_{I}}=&(a,k_{i_{1}},\dots,k_{i_{n}})\ot(\alpha_{0}-a+\sum_{j=1}^{n}(\alpha_{i_{j}}-k_{i_{j}}),\alpha_{u_{1}},\dots, \alpha_{u_{p-n}}),\\
 	\{u_{1}<\dots<u_{p-n}\}=&\llbracket1,p\rrbracket\setminus\{i_{1},\dots,i_{n}\}.
 	\end{align*}
 \end{prop} 
 \Ex{} Donnons quelques exemples pour le calcul du coproduit de la cogèbre $\Cp$. On ne donne que les coproduits réduits.
 \begin{align*}
	 \tilde{\De}((1,1,1))=&(1)\ot(0,1,1)+2(0,1)\ot(1,1)+(0,1,1)\ot(1)+2(1,1)\ot(0,1),\\
	 \tilde{\De}((0,2,1))=&(0,1)\ot(0,2) +2(0,1)\ot(1,1)+(0,2)\ot(0,1)+2(0,1,1)\ot(1).
 \end{align*}
 \Rq{} 
 \begin{enumerate}
 	\item La cogèbre $\Cp$ n'est pas colibre. En effet, 
 	\begin{align*}
 	\Prim{\Cp}{1}=& Vect((1),(0,1)),\\
 	\Prim{\Cp}{2}=& (0),\\
 	\Prim{\Cp}{3}=& Vect((0,1,2)-(0,2,1))
 	\end{align*}
 	Or, la série formelle de $\Cp$ étant égale à $F_{\Cp}(h)=\sum\limits_{n=0}^{\infty}2^{n}h^{n}=\frac{1}{1-2h}$, on a $1-\frac{1}{F_{\Cp}(h)}=2h$.
 	Donc, si $\Cp$ était colibre, on aurait $\Prim{\Cp}{3}=(0)$; ce qui n'est pas le cas.
 	\item Les seules compositions étendues primitives sont celles de degré 1, \emph{i.e.} $(1)$ et $(0,1)$.
 \end{enumerate}

\paragraph{Structure de coproduit semi-direct.}
L'objectif de ce paragraphe est de décrire la structure d'algèbre de Hopf de $\Cp$ en fonction de celle de l'algèbre de polynômes $\K[(1)]$ et de celle de l'algèbre tensorielle $\tens{(0,n),n\in\N^{*}}$. Pour cela, nous allons établir un isomorphisme entre $\Cp$  et le coproduit semi-direct d'algèbres de Hopf $\K[(1)]\rtimes\tens{(0,n),n\in\N^{*}}$.

\subparagraph{Rappels.} Commençons par quelques rappels sur le coproduit semi-direct d'algèbres de Hopf (\emph{cf.} \cite{Molnar1977}).
Considérons donc $(H,m_{H},\eta_{H},\De_{H},\varepsilon_{H})$ et $(C,m_{C},\eta_{C},\De_{C},\varepsilon_{C})$ deux algèbres de Hopf. On suppose que $C$ possède une structure de $H$-comodule à droite donnée par le morphisme: $$\rho:\left\{\begin{array}{rcl} C &\longrightarrow & C\ot H \\ c &\longrightarrow & \sum c_{(1)}\ot c_{(2)}.\end{array}\right.$$
On définit la volte $\nu$ par:
$$\nu:\left\{\begin{array}{rcl} H\ot C &\longrightarrow & C\ot H \\ h\ot c &\longrightarrow & c\ot h.\end{array}\right.$$
On appelle $Id_{H}$ (respectivement $Id_{C}$) la fonction identité de $H$ (respectivement $C$).
\begin{defi}
	\begin{enumerate}
		\item $C$ est une comodule-algèbre si $m_{C}$ et $\eta_{C}$ sont des morphismes de $H$-comodules \emph{i.e.} $$\rho\circ\eta_{C}=(\eta_{C}\ot\eta_{H})\circ \rho_{|_{\K}} \text{ et } \rho\circ m_{C}=(m_{C}\ot m_{H})\circ(Id_{C}\ot\nu\ot Id_{H})\circ(\rho\ot\rho).$$
		\item $C$ est une comodule-cogèbre si $\De_{C}$ et $\varepsilon_{C}$ sont des morphismes de $H$-comodules \emph{i.e.} $$(\varepsilon_{C}\ot Id_{H})\circ\rho=\eta_{H}\circ\varepsilon_{C} \text{ et } (\De_{C}\ot Id_{H})\circ\rho=(Id_{C}\ot Id_{C}\ot m_{H})\circ(Id_{C}\ot\nu\ot Id_{H})\circ(\rho\ot\rho)\circ\De_{C}.$$
		\item $C$ est une comodule-bigèbre si elle est à la fois une comodule-algèbre et une comodule-cogèbre.
	\end{enumerate}
\end{defi}
Sur $H\ot C$ on définit les morphismes suivants:
\begin{align*}
\overline{m}=&(m_{H}\ot m_{C})\circ(Id_{H}\ot\nu^{-1}\ot Id_{C}),\\
\overline{\eta}=&\eta_{H}\ot\eta_{C},\\
\overline{\varepsilon}=&\varepsilon_{H}\ot \varepsilon_{C},\\
\overline{\De}=&(Id_{H}\ot Id_{C} \ot m_{H} \ot Id_{C})\circ (Id_{H}\ot \nu \ot Id_{H} \ot Id_{C})\circ (\De_{H}\ot\rho \ot Id_{C} )\circ(Id_{H}\ot \De_{C}).
\end{align*}
\begin{theo}\label{coprodsemidirect}Supposons que $H$ soit une algèbre de Hopf commutative et que l'algèbre de Hopf $C$ soit une comodule-bigèbre. On obtient alors que $(H\ot C,\overline{m},\overline{\eta},\overline{\De},\overline{\varepsilon})$ est une algèbre de Hopf notée $H\rtimes C$ et appelée coproduit semi-direct de $C$ par $H$.
\end{theo}
\begin{proof}
	Il s'agit du théorème 2.14 énoncé et démontré par Molnar dans \cite{Molnar1977}.
\end{proof}

\subparagraph{Interpretation de $\Cp$ en terme d'un coproduit semi-direct d'algèbres de Hopf.}
Interprétons $\Cp$ comme le coproduit semi-direct d'une algèbre de Hopf $C$ par une algèbre de Hopf $H$. Pour cela posons $H=\K[(1)]$ et $C=\tens{(0,n),n\in\N^{*}}$. Commençons par expliciter les opérations de ces deux algèbres de Hopf.
\begin{description}
	\item[Cas de $H$.] L'algèbre de Hopf $H$ est isomorphe à l'algèbre de polynômes $\K[X]$. Son produit $\ast_{H}$ est donc donné par le produit habituel des polynômes \emph{i.e.}:
	$$\ast_{H}:\left\{\begin{array}{rcl} H\ot H &\longrightarrow & H \\ (k)\ot(l) &\longrightarrow& (k)\ast_{H}(l)=(1)^{k}\ast_{H}(1)^{l}=(k+l).\end{array}\right.$$
	Le coproduit $\De_{H}$ rend $(1)$ primitif. Il est ainsi complètement déterminé.
	\Ex{}
	\begin{align*}
		[3(2)-4(1)+6]\ast_{H}[(1)-1]=&[3(1)^{2}-4(1)+6]\ast_{H}[(1)-1]=3(3)-7(2)+10(1)-6,\\
		\De_{H}((3)-2(1)+4)=&[(3)-2(1)]\ot1_{H}+ 1_{H}\ot[(3)-2(1)]\\
		+& 4~1_{H}\ot 1_{H}+ 3(1)\ot(2)+3(2)\ot(1).
	\end{align*}
	\item[Cas de $C$.] L'algèbre de Hopf $C$ est isomorphe à celle des fonctions symétriques non commutatives $\NSym$. Son produit $\ast_{C}$ est donc donné par concaténation \emph{i.e.}:
	$$\ast_{C}:\left\{\begin{array}{rcl} C\ot C &\longrightarrow & C \\ (0,n_{1},\dots,n_{k})\ot(0,m_{1},\dots,m_{s}) &\longrightarrow& (0,n_{1},\dots,n_{k},m_{1},\dots,m_{s}).\end{array}\right.$$
	Son coproduit $\De_{C}$ est complètement déterminé en rendant primitifs ses générateurs \emph{i.e.}:
	$$\De_{C}:\left\{\begin{array}{rcl} C &\longrightarrow & C\ot C \\ (0,n), n\in\N^{*} &\longrightarrow& (0,n)\ot1_{C}+1_{C}\otimes(0,n).\end{array}\right.$$
	\Ex{}
	\begin{align*}
		(0,2,3,1)\ast_{C}(0,9,4,6)=&(0,2,3,1,9,4,6),\\
		\De_{C}((0,1,2))=&(0,1,2)\ot1_{C}+(0,1)\ot(0,2)+ (0,2)\ot(0,1)+1_{C}\otimes(0,1,2).
	\end{align*}
\end{description}
\begin{prop}\label{Hcomodule}
	On définit l'application suivante:
	$$\rho:\left\{\begin{array}{rcl} C &\longrightarrow & C\ot H \\(0,n_{1},\dots,n_{q}) &\longrightarrow& \displaystyle\sum\limits_{s=q}^{n_{1}+\dots+n_{q}-1}\sum\limits_{\substack{ k_{1}+\dots+k_{q}=s \\ 1\leq k_{j}\leq n_{j}}}\binom{n_{1}}{k_{1}}\dots\binom{n_{q}}{k_{q}}(0,k_{1},\dots,k_{q})\ot(\sum\limits_{j=1}^{q}n_{j}-k_{j})\\
	&&+(0,n_{1},\dots,n_{q})\ot1_{H}.\end{array}\right.$$
	$(C,\rho)$ est un $H$-comodule à droite.
\end{prop}
\Ex{}
\begin{align*}
\rho((0,n))=&\sum_{k=1}^{n-1}\binom{n}{k}(0,k)\ot(n-k)+(0,n)\ot1_{H} \text{ pour } n\in\N^{*},\\
\rho((0,2,2))=&4(0,1,1)\ot(2)+2(0,1,2)\ot(1)+2(0,2,1)\ot(1)+ (0,2,2)\ot1_{H}.
\end{align*}
\begin{proof}
	Soit $(0,n_{1},\dots,n_{q})\in C\setminus\{1_{C}\}$. Calculons $\rho((0,n_{1},\dots,n_{q}))$ à partir de $\rho((0,n_{j}))$, pour $j\in\{1,\dots,q\}$. Posons d'abord $\theta=n_{1}+\dots+n_{q}$. On a:
	\begin{align*}
	\rho((0,n_{1}))\dots\rho((0,n_{q}))=&\prod_{j=1}^{\underrightarrow{q}}\left[\sum_{k_{j}=1}^{n_{j}-1}\binom{n_{j}}{k_{j}}(0,k_{j})\ot(n_{j}-k_{j})+(0,n_{j})\ot1_{H}\right]\\
	=&\sum_{s=q}^{\theta-1}\sum_{\substack{k_{1}+\dots+k_{q}=s\\ 1\leq k_{j}\leq n_{j}}}\binom{n_{1}}{k_{1}}\dots\binom{n_{q}}{k_{q}}(0,k_{1},\dots,k_{q})\ot(\theta-s)\\
	&+(0,n_{1},\dots,n_{q})\ot1_{H}\\
	=&\rho((0,n_{1},\dots,n_{q})).
	\end{align*}
	Ainsi $\rho\circ m_{C}=(m_{C}\ot m_{H})\circ(Id_{C}\ot\nu\ot Id_{H})\circ(\rho\ot\rho)$ \emph{i.e.} $\rho$ est un morphisme d'algèbres de $C$ dans $C\ot H$. Donc, pour montrer que $(Id_{C}\ot \De_{H})\circ\rho=(\rho\ot Id_{H})\circ\rho$ et que $(Id_{C}\ot \varepsilon_{H})\circ\rho=Id_{C}$ il suffit de le vérifier pour les éléments $(0,n)$ tels que $n\in\N^{*}$.
	Soit $(0,n)\in C\setminus\{1_{C}\}$. On a alors:
	\begin{align*}
	(Id_{C}\ot \De_{H})\circ\rho((0,n))=&\sum_{k=1}^{n-2}\sum_{l=1}^{n-k-1}\binom{n}{k}\binom{n-k}{l}(0,k)\ot(l)\ot(n-k-l)\\
	+&\sum_{k=1}^{n-1}\binom{n}{k}(0,k)\ot1_{H}\ot(n-k)+\sum_{k=1}^{n-1}\binom{n}{k}(0,k)\ot(n-k)\ot1_{H}\\
	+&(0,n)\ot1_{H}\ot1_{H}
	\end{align*}
	et  
	\begin{align*}
	(\rho\ot Id_{H})\circ\rho((0,n))=&\sum_{k=2}^{n-1}\sum_{l=1}^{k-1}\binom{k}{l}\binom{n}{k}(0,l)\ot(k-l)\ot(n-k)\\
	+&\sum_{k=1}^{n-1}\binom{n}{k}(0,k)\ot1_{H}\ot(n-k)+\sum_{s=1}^{n-1}\binom{n}{s}(0,s)\ot(n-s)\ot1_{H}\\
	+&(0,n)\ot1_{H}\ot1_{H}.
	\end{align*}
	On obtient donc $(\rho\ot Id_{H})\circ\rho=(Id_{C}\otimes\De_{H})\circ\rho$.
	
	De plus,
	$$(Id_{C}\ot \varepsilon_{H})\circ\rho((0,n))=\sum_{k=1}^{n-1}\binom{n}{k}(0,k)\varepsilon_{H}((n-k))+(0,n)\varepsilon_{H}(1_{H})
	=(0,n);$$
	donc $(Id_{C}\ot \varepsilon_{H})\circ\rho=Id_{C}$.
\end{proof}

\begin{prop}\label{bHcomodule}
	$C$ est une comodule-bigèbre.
\end{prop}
\begin{proof}
	\begin{description}
		\item[Structure de comodule-algèbre.] 
		$\rho(1_{C})=1_{C}\ot1_{H}$ par définition. Dans la démonstration de la proposition \ref{Hcomodule}, nous avons déjà prouvé l'égalité entre les morphismes $\rho\circ m_{C}$ et $(m_{C}\ot m_{H})\circ(Id_{C}\ot\nu\ot Id_{H})\circ(\rho\ot\rho)$.
		\item[Structure de comodule-cogèbre.] 
		Il est immédiat que $(\varepsilon_{C}\ot Id_{H})\circ\rho=\eta_{H}\circ\varepsilon_{C}$. Soit $(0,n)\in C\setminus\{1_{C}\}$. On a alors:
		\begin{align*}
			(\De_{C}\ot Id_{H})\circ\rho((0,n))=&\sum_{k=1}^{n-1}\binom{n}{k}(0,k)\ot1_{C}\ot(n-k)+\sum_{k=1}^{n-1}\binom{n}{k}1_{C}\ot(0,k)\ot(n-k)\\
			+&(0,n)\ot1_{C}\ot1_{H}+1_{C}\ot(0,n)\ot1_{H}
		\end{align*}
		et
		\begin{align*}
			g((0,n))=&(Id_{C}\ot Id_{C}\ot m_{H})\circ(Id_{C}\ot\nu\ot Id_{H})\circ(\rho\ot\rho)\circ\De_{C}((0,n))\\
			=&\sum_{k=1}^{n-1}\binom{n}{k}(0,k)\ot1_{C}\ot(n-k)+(0,n)\ot1_{C}\ot1_{H}\\
			+&\sum_{k=1}^{n-1}\binom{n}{k}1_{C}\ot(0,k)\ot(n-k)+1_{C}\ot(0,n)\ot1_{H}.
		\end{align*}
		Donc, $(\De_{C}\ot Id_{H})\circ\rho((0,n))=(Id_{C}\ot Id_{C}\ot m_{H})\circ(Id_{C}\ot\nu\ot Id_{H})\circ(\rho\ot\rho)\circ\De_{C}((0,n))$. Comme $\rho$ et $\De_{C}$ sont des morphismes d'algèbres, l'égalité précédente suffit pour montrer l'égalité entre $(\De_{C}\ot Id_{H})\circ\rho$ et $(Id_{C}\ot Id_{C}\ot m_{H})\circ(Id_{C}\ot\nu\ot Id_{H})\circ(\rho\ot\rho)\circ\De_{C}$.
	\end{description}
\end{proof}
D'après le théorème \ref{coprodsemidirect} page \pageref{coprodsemidirect}, on peut considérer l'algèbre de Hopf coproduit semi-direct $H\rtimes C$.
\begin{theo}\label{isomCoprodSemidirect}
	L'application 
	$$\Upsilon:\left\{\begin{array}{rcl} \Cp &\longrightarrow & H\rtimes C \\(\alpha_{0},\alpha_{1},\dots,\alpha_{k}) &\longrightarrow& (\alpha_{0})\otimes(0,\alpha_{1},\dots,\alpha_{k})\end{array}\right.$$ est un isomorphisme d'algèbres de Hopf.
\end{theo}
\begin{proof}
	Le fait que $\Upsilon$ soit un isomorphisme d'algèbres est immédiat. Dans cette démonstration notons $\De_{\Cp}$ (respectivement $\overline{\De}$) le coproduit de $\Cp$ (respectivement le coproduit de $H\rtimes C$). Montrons que $(\Upsilon\otimes\Upsilon)\circ\De_{\Cp}=\overline{\De}\circ\Upsilon$. Comme $\Upsilon$, $\De_{\Cp}$ et $\overline{\De}$ sont des morphismes d'algèbres, il suffit de le montrer pour $(1)$ et tout élément de $\Cp$ de la forme $(0,n)$ où $n\in\N^{*}$. Pour $(1)$, c'est immédiat. Soit maintenant $n$ un entier naturel non nul. 
	\begin{align*}
		\overline{\De}\circ\Upsilon((0,n))=&\sum_{k=1}^{n-1}\binom{n}{k}[1_{H}\otimes(0,k)]\ot[(n-k)\otimes1_{C}]\\
		+&[1_{H}\otimes(0,n)]\ot[1_{H}\otimes1_{C}]+[1_{H}\otimes1_{C}]\ot[1_{H}\otimes(0,n)]\\
		=&\sum_{k=1}^{n-1}\binom{n}{k}\Upsilon\otimes\Upsilon((0,k)\ot(n-k))+\Upsilon\otimes\Upsilon((0,n)\ot1_{\Cp})+\Upsilon\otimes\Upsilon(1_{\Cp}\ot(0,n))\\
		=&(\Upsilon\otimes\Upsilon)\circ\De_{\Cp}((0,n)).
	\end{align*}
\end{proof}

\paragraph{Actions liées à la structure de coproduit semi-direct.}
Pour définir la structure de coproduit semi-direct de $C=\tens{(0,n),n\in\N^{*}}$ sur $H=\K[(1)]$, nous avons eu recours à une coaction à droite $\rho$, à savoir:
$$\rho:\left\{\begin{array}{rcl} C &\longrightarrow & C\ot H \\(0,n_{1},\dots,n_{q}) &\longrightarrow& \displaystyle\sum\limits_{s=q}^{n_{1}+\dots+n_{q}-1}\sum\limits_{\substack{ k_{1}+\dots+k_{q}=s \\ 1\leq k_{j}\leq n_{j}}}\binom{n_{1}}{k_{1}}\dots\binom{n_{q}}{k_{q}}(0,k_{1},\dots,k_{q})\ot(\sum\limits_{j=1}^{q}n_{j}-k_{j})\\
&&+(0,n_{1},\dots,n_{q})\ot1_{H}.\end{array}\right.$$
Cette coaction permet de mettre à jour deux actions: l'une de $\dual{H}$ sur $\dual{C}$, l'autre des caractères de $H$ sur ceux de $C$.

\subparagraph{Action duale de $\rho$.}
Transposons l'opération $\rho$ afin d'obtenir une action à gauche de $\dual{H}$ sur $\dual{C}$. Avant de décrire cette action, notons d'abord $(Z_{(0,n_{1},\dots,n_{s})})_{(s,(n_{1},\dots,n_{s}))\in\N^{*}\times(\N^{*})^{s}}$ et $(Z_{(k)})_{k\in\N^{*}}$ les bases duales respectives de $\dual{C}$ et $\dual{H}$.
\begin{prop}\label{actionDuale}
	L'action duale de $\rho$ est donnée par le morphisme:
	$$\rho^{*}:\left\{\begin{array}{rcl} \dual{C}\ot\dual{H}  &\longrightarrow & \dual{C}\\  Z_{(0,n_{1},\dots,n_{s})} \ot Z_{(k)} &\longrightarrow &\displaystyle\sum\limits_{\substack{\delta_{1}+\dots+\delta_{s}=k\\ \delta_{i}\in\N}}\binom{n_{1}+\delta_{1}}{n_{1}}\dots\binom{n_{s}+\delta_{s}}{n_{s}}Z_{(0,n_{1}+\delta_{1},\dots,n_{s}+\delta_{s})},\\ 
	1_{\dual{C}}\ot Z_{(k)} &\longrightarrow &0. 
	\end{array}\right.$$
\end{prop}
\begin{proof}
	On le démontre par calcul direct.
\end{proof}

\Ex{}
\begin{align*}
	\rho^{*}(Z_{(0,n_{1},\dots,n_{s})}\ot 1_{\dual{H}})=&Z_{(0,n_{1},\dots,n_{s})},\\
	\rho^{*}(Z_{(0,5,23,4)}\ot Z_{(1)})=&Z_{(0,6,23,4)}+Z_{(0,5,24,4)}+Z_{(0,5,23,5)},\\
	\rho^{*}(Z_{(0,5,23,4)}\ot Z_{(2)})=&Z_{(0,7,23,4)}+Z_{(0,5,25,4)}+Z_{(0,5,23,6)}+Z_{(0,6,24,4)}+Z_{(0,6,23,5)}+Z_{(0,5,24,5)}.
\end{align*}

\subparagraph{Action des caractères de $\K[(1)]$ sur ceux de $\tens{(0,n),n\in\N^{*}}$.}
Notons $Char(H)$ et $Char(C)$ les caractères de $H$ et $C$ respectivement. On désigne par $\K[[X]]_{+}$ l'espace vectoriel des séries formelles sans terme constant.
\begin{prop}
	Définissons les morphismes:
		$$\omega_{H}:\left\{\begin{array}{rcl} (\K,+) &\longrightarrow & (Char(H),+) \\ \lambda &\longrightarrow & \left\{\begin{array}{rcl} H &\longrightarrow & \K \\ (1) &\longrightarrow &\lambda,\end{array}\right. \end{array}\right.$$
		
		$$\omega_{C}:\left\{\begin{array}{rcl} (\K[[X]]_{+},+) &\longrightarrow & (Char(C),+) \\ \sum\limits_{n=1}^{\infty}a_{n}X^{n} &\longrightarrow & \left\{\begin{array}{rcl} C &\longrightarrow & \K \\ (0,s) &\longrightarrow & s!a_{s},\end{array}\right. \end{array}\right.$$
		et 
		$$\overline{\omega}_{C}:\left\{\begin{array}{rcl} (\K[[X]]_{+},+) &\longrightarrow & (Char(C),+) \\ \sum\limits_{n=1}^{\infty}a_{n}X^{n} &\longrightarrow & \left\{\begin{array}{rcl} C &\longrightarrow & \K \\ (0,s) &\longrightarrow & \displaystyle\frac{a_{s}}{s!}.\end{array}\right. \end{array}\right.$$
	Ces trois applications sont des isomorphismes de groupes.
\end{prop}

\begin{proof}
	Direct.
\end{proof}

Grâce aux isomorphismes $\omega_{H}$, $\omega_{C}$ et $\overline{\omega}_{C}$, donner une action de $Char(H)$ sur $Char(C)$ revient à donner une action de $\K$ sur $\K[[X]]_{+}$.

\begin{prop}\label{actionCaracteres}
	Soit $\Omega$ l'application:
	$$\Omega:\left\{\begin{array}{rcl} \K\times\K[[X]]_{+} &\longrightarrow & \K[[X]]_{+} \\ (\lambda,a) &\longrightarrow & \displaystyle\sum_{s=1}^{\infty}\overline{\omega}_{C}\Bigg(\sum\limits_{n=1}^{\infty}[(\omega_{C}(a)\ot\omega_{H}(\lambda))\circ \rho]((0,n))X^{n}\Bigg)((0,s))X^{s} \end{array}\right.$$ où $\rho$ désigne la coaction $\rho: C \longrightarrow C\ot H$ définie dans le paragraphe précédent.
	Il s'agit de l'application $$\Omega:\left\{\begin{array}{rcl} \K\times\K[[X]]_{+} &\longrightarrow & \K[[X]]_{+} \\ (\lambda,a) &\longrightarrow & ae^{\lambda X}. \end{array}\right.$$ Elle définit donc une action de $\K$ sur $\K[[X]]_{+}$ et, par suite, une action de $Char(H)$ sur $Char(C)$.
\end{prop}
\begin{proof}
	Soient $\lambda$ un élément de $\K$, $a$ une série formelle sans terme constant et $n$ un entier naturel non nul. Posons $b=\sum\limits_{s=1}^{\infty}b_{s}X^{s}=\Omega(\lambda,a)$. On a alors:
	$$(\omega_{C}(a)\ot\omega_{H}(\lambda))\circ \rho((0,n))=\sum_{k=1}^{n}\binom{n}{k}k!a_{k}\lambda^{n-k}=n!\sum_{k=1}^{n}a_{k}\frac{\lambda^{n-k}}{(n-k)!}.$$
	Or, $$(\omega_{C}(a)\ot\omega_{H}(\lambda))\circ \rho((0,n))=n!b_{n}.$$
	On obtient donc que $b_{n}=\langle ae^{\lambda X},X^{n}\rangle$. Ainsi $b=ae^{\lambda X}$ et l'application $\Omega$ définit une action.
\end{proof}

Intéressons-nous maintenant aux éléments primitifs de $\Cp$ et déterminons-en quelques propriétés. 
\paragraph{Etude des éléments primitifs de $\Cp$.}
\begin{prop}\label{primCPzero}
	Soient $n\in\N$ avec $n\geq2$, $k\in\N$ et $u=\sum\limits_{s=1}^{k}\mu_{s}(\beta_{s,0},\dots,\beta_{s,p_{s}})$ un élément primitif de $(\Cp)_{n}$. Supposons qu'il existe un entier $s\in\{1,\dots,k\}$ tel que $(\beta_{s,0}\neq 0\text{ et } p_{s}\geq1)$ ou $(\beta_{s,0}\geq 2\text{ et } p_{s}=0)$. On a alors: $\mu_{s}=0$.
\end{prop}

\begin{proof}
	Soient $n\in\N$ avec $n\geq2$, $k\in\N$ et $u=\sum\limits_{s=1}^{k}\mu_{s}(\beta_{s,0},\dots,\beta_{s,p_{s}})$ un élément de $\Prim{\Cp}{n}$. 
	\begin{description}
		\item[Cas $\beta_{s,0}\neq 0\text{ et } p_{s}\geq1$.] Le coefficient $c_{s}$ de $(1)\ot(\beta_{s,0}-1,\beta_{s,1},\dots,\beta_{s,p_{s}})$ dans $\De(u)$ est donc nul. Or $c_{s}=\beta_{s,0}\mu_{s}$. Donc $\mu_{s}=0$.
		\item[Cas $\beta_{s,0}\geq 2\text{ et } p_{s}=0$.] Le coefficient $c_{s}$ de $(1)\ot(\beta_{s,0}-1)$ dans $\De(u)$ est donc nul. Or $c_{s}=\beta_{s,0}\mu_{s}$. Donc $\mu_{s}=0$.
	\end{description}
\end{proof}

Pour obtenir des informations supplémentaires sur $Prim(\Cp)$, regardons, pour la proposition suivante, $\Cp$ sous sa version $\WMat/\Idl{J}$. Dans $\WMat/\Idl{J}$, on identifie les classes avec les mots tassés $\underbrace{x_{0}\dots x_{0}}_{\alpha_{0} \text{ fois}}\dots\underbrace{x_{k}\dots x_{k}}_{\alpha_{k} \text{ fois}}$ ($k,\alpha_{0}\in\N$, $\alpha_{1},\dots,\alpha_{k}\in\N^{*}$).

\begin{prop}\label{lienCpISPW}
	Soient $k\in\N^{*}$ et $u=\sum\limits_{s=1}^{k}\mu_{s}\underbrace{x_{1}\dots x_{1}}_{\beta_{s,1} \text{ fois}}\dots\underbrace{x_{p_{s}}\dots x_{p_{s}}}_{\beta_{s,p_{s}} \text{ fois}}$ un élément de $\Cp$.
	$$u\in Prim(\WMat/\Idl{J})\Rightarrow u\in Prim(\ISPW).$$
\end{prop}

\begin{proof} Dans la démonstration, les coproduits de $\WMat/\Idl{J}$ et $\ISPW$ seront notés $\De_{\WMat/\Idl{J}}$ et $\De_{\ISPW}$ respectivement.
Soient $k\in\N^{*}$ et $u=\sum\limits_{s=1}^{k}\mu_{s}\underbrace{x_{1}\dots x_{1}}_{\beta_{s,1} \text{ fois}}\dots\underbrace{x_{p_{s}}\dots x_{p_{s}}}_{\beta_{s,p_{s}} \text{ fois}}$ un élément de $\Cp$. Supposons que  $u$ soit primitif dans $\WMat/\Idl{J}$. Alors, pour tout $s\in\{1,\dots,k\}$ et tout $k$-uplet $(\epsilon_{1},\dots,\epsilon_{k})\in(\{0,1\})^{k}\setminus\{(0,\dots,0),(1,\dots,1)\}$, le coefficient du tenseur $$(\underbrace{x_{1}\dots x_{1}}_{\beta_{s,1} \text{ fois}})^{\epsilon_{1}}\ast\dots\ast(\underbrace{x_{1}\dots x_{1}}_{\beta_{s,p_{s}} \text{ fois}})^{\epsilon_{p_{s}}}\ot (\underbrace{x_{1}\dots x_{1}}_{\beta_{s,1} \text{ fois}})^{1-\epsilon_{1}}\ast\dots\ast(\underbrace{x_{1}\dots x_{1}}_{\beta_{s,p_{s}} \text{ fois}})^{1-\epsilon_{p_{s}}}$$ 
dans $\De_{\WMat/\Idl{J}}(w)$ est nul. Or, dans $\ISPW$, le coproduit de $u$ est égal à : 
$$\De_{\ISPW}(u)=\sum_{s=1}^{k}\sum_{l=0}^{p_{s}}\sum_{1\leq i_{1}<\dots<i_{l}\leq p_{s}}\mu_{s}\underbrace{x_{1}\dots x_{1}}_{\beta_{s,i_{1}} \text{ fois}}\dots\underbrace{x_{l}\dots x_{l}}_{s,\beta_{i_{l}} \text{ fois}}\ot \underbrace{x_{1}\dots x_{1}}_{\beta_{s,u_{1}} \text{ fois}}\dots\underbrace{x_{p_{s}-l}\dots x_{p_{s}-l}}_{\beta_{s,u_{p_{s}-l}} \text{ fois}}$$ où $\{u_{1}<\dots<u_{p_{s}-l}\}=\llbracket1,p_{s}\rrbracket\setminus\{i_{1},\dots,i_{l}\}$.
$u$ est donc primitif dans $\ISPW$.
\end{proof}
\Ex{1} Considérons l'élément $u=(0,1,1,2)-2(0,1,2,1)+(0,2,1,1)$. Il est primitif dans $\Cp$. Dans $\ISPW$, $u$ correspond à l'élément $w=x_{1}x_{2}x_{3}x_{3}-2x_{1}x_{2}x_{2}x_{3}+x_{1}x_{1}x_{2}x_{3}$. D'après la proposition \ref{prim}, $w$ est bien primitif dans $\ISPW$.\\ 

L'objectif maintenant est donc d'étudier la réciproque de la proposition précédente. Elle s'avère être fausse. En effet, si l'on considère les éléments primitifs de $\ISPW$ établis dans les propositions \ref{prim} et \ref{prim_Lambda}, seule une sous-famille particulière permet d'obtenir des éléments primitifs dans $\Cp$. Pour le démontrer, définissons les éléments de $\Cp$ induits par les propositions \ref{prim} et \ref{prim_Lambda}.

\begin{defi}
	Soient $n\in\N^{*}$, $(\alpha_{1},\dots,\alpha_{n})\in(\N^{*})^{n}$ et $(\beta_{1},\dots,\beta_{n})\in(\N^{*})^{n}$ un $n$-uplet d'entiers deux à deux distincts. On note $\rho=\alpha_{2}+\dots+\alpha_{n}$ et $\theta=\alpha_{1}+\dots+\alpha_{n}$. On appelle $\gamma$ le $\theta$-uplet $\gamma=(\underbrace{\beta_{1},\dots,\beta_{1}}_{\alpha_{1} \text{ fois}},\dots,\underbrace{\beta_{n},\dots,\beta_{n}}_{\alpha_{n} \text{ fois}})$. 
	
	On note $\Gamma_{\gamma}$ et $\Gamma_{\Lambda_{\gamma}}$ les éléments de $\Cp$ définis par:
	
	$$\Gamma_{\gamma}=\frac{1}{\alpha_{2}!\dots\alpha_{n}!}\sum_{\sigma\in\Sn{\rho}}\sum_{k=1}^{\alpha_{1}}(-1)^{k-1}\binom{\alpha_{1}-1}{k-1}\sum_{s=0}^{\rho}(-1)^{s}\binom{\rho}{s}u_{\gamma,\tilde{\sigma}_{k,s}}$$
	où
	\begin{align*}
	\tilde{\sigma}_{1,0}=&\begin{pmatrix}1 & \dots & \alpha_{1} & \alpha_{1}+1 & \dots & \theta \\1 & \dots & \alpha_{1} & \sigma(1)+\alpha
	_{1} & \dots & \sigma(\rho)+\alpha_{1}\end{pmatrix}\in\Sn{\theta},\\
	\tilde{\sigma}_{k,0}=&\begin{pmatrix}1 & \dots & \alpha_{1}-k+1 & \alpha_{1}-k+2 & \dots & \theta-k+1 & \theta-k+2 & \dots & \theta \\1 & \dots & \alpha_{1}-k+1 & \sigma(1)+\alpha_{1} & \dots & \sigma(\rho)+\alpha_{1} & \alpha_{1}-k+2 & \dots & \alpha_{1}\end{pmatrix}\in\Sn{\theta},\\
	\tau_{k,s}=& \begin{pmatrix}\alpha_{1}-k+s & \alpha_{1}-k+s+1 \end{pmatrix}\in\Sn{\theta} \text{ pour } 1\leq k \leq \alpha_{1} \text{ et } 1\leq s \leq \rho,\\
	\tilde{\sigma}_{k,s}=&\tilde{\sigma}_{k,s-1}\circ\tau_{k,s}\in\Sn{\theta} \text{ pour } 1\leq k \leq \alpha_{1} \text{ et } 1\leq s \leq \rho,\\
	u_{\gamma,\tilde{\sigma}_{k,s}}=&(0,\gamma_{\tilde{\sigma}_{k,s}(1)},\dots,\gamma_{\tilde{\sigma}_{k,s}(\theta)}),
	\end{align*}
	
	et

 $$\Gamma_{\Lambda_{\gamma}}=\frac{1}{\alpha_{1}!\dots\alpha_{n}!}\sum_{\sigma\in\Sn{\theta}}\sum_{i=1}^{\alpha_{1}}(-1)^{\sigma^{-1}(i)-1}\binom{\theta-1}{\sigma^{-1}(i)-1}(0,\gamma_{\sigma(1)},\dots,\gamma_{\sigma(\theta)}).$$ 
\end{defi}

Pour énoncer la proposition suivante introduisons une nouvelle notation. Soient $\kappa$ un entier non nul et $\gamma=(0,\gamma_{1},\dots,\gamma_{\kappa})$ une composition étendue. On définit l'ensemble d'éléments distincts $L_{\gamma}$ par: $$L_{\gamma}=\big\{ (0,\gamma_{\sigma(1)},\dots,\gamma_{\sigma(\kappa)}), \sigma\in\Sn{\kappa}\big \}.$$

\begin{prop}\label{conditionEcartMax}
	Soient $k$ un entier naturel non nul, $(\alpha_{1},\dots,\alpha_{k})$ et $(\beta_{1}<\dots<\beta_{k})$ deux $k$-uplets d'entiers naturels non nuls, $\beta=(0,\underbrace{\beta_{1},\dots ,\beta_{1}}_{\alpha_{1} \text{ fois}},\dots,\underbrace{\beta_{k},\dots ,\beta_{k}}_{\alpha_{k} \text{ fois}})$ une composition étendue et $u=\sum\limits_{\omega\in L_{\beta}}\mu_{\omega}\omega$ un élément de $\Cp$. Supposons que $u$ soit un élément de $Prim(\Cp)$. Posons $\beta_{0}=0$. S'il existe un entier $i\in\llbracket1,k\rrbracket$ vérifiant $\beta_{i}\geq \beta_{i-1}+2$, alors, l'élément $u$ est nul.	
\end{prop}

\begin{proof}
	Soient $k\in\N^{*}$, $(\alpha_{1},\dots,\alpha_{k})\in(\N^{*})^{k}$, $\beta=(0,\underbrace{\beta_{1},\dots ,\beta_{1}}_{\alpha_{1} \text{ fois}}<\dots<\underbrace{\beta_{k},\dots ,\beta_{k}}_{\alpha_{k} \text{ fois}})$ et $u=\sum\limits_{\omega\in L_{\beta}}\mu_{\omega}\omega$ des éléments de $\Cp$. Posons $\beta_{0}=0$. Supposons que $u$ soit primitif et qu'il existe $i\in\llbracket1,k\rrbracket$ tel que $\beta_{i}\geq \beta_{i-1}+2$. Ecrivons $\beta$ sous la forme $\beta=(0,\gamma_{1},\dots,\gamma_{\alpha_{1}+\dots+\alpha_{k}})$ avec  $\gamma_{j}=\beta_{l}$ où $l$ est l'unique entier tel que $j\in\{(\alpha_{1}+\dots+\alpha_{l-1}+1),\dots, (\alpha_{1}+\dots+\alpha_{l})\}$. Posons $M=\max\big\{l\in\{1,\dots, (\alpha_{1}+\dots+\alpha_{k})\}, \gamma_{l}=\beta_{i} \big\}$. 
	
	Soit $\omega\in L_{\beta}$. Il existe $\sigma\in\Sn{\alpha_{1}+\dots+\alpha_{k}}$ tel que $\omega=(0,\gamma_{\sigma(1)},\dots,\gamma_{\sigma(\alpha_{1}+\dots+\alpha_{k})})$.	Comme $u$ est primitif, le coefficient $c_{\omega}$ de $(0,\gamma_{\sigma(1)},\dots,\gamma_{M}-1,\dots,\gamma_{\sigma(\alpha_{1}+\dots+\alpha_{k})})\ot (1)$ dans $\De(u)$ est nul. Or $c_{\omega}=\beta_{i}\mu_{\omega}$. Donc $\mu_{\omega}=0$. Comme ceci est vrai pour tout $\omega\in L_{\beta}$, on obtient que $u=0$.  
\end{proof}

\begin{prop}
	Soit $n$ un entier naturel non nul, $(\alpha_{1},\dots,\alpha_{n})$ un $n$-uplet d'entiers naturels non nuls et $(\beta_{1},\dots,\beta_{n})$ un $n$-uplet d'entiers naturels non nuls deux à deux distincts. On pose $\theta=\displaystyle\sum_{i=1}^{n}\alpha_{i}$ et on appelle $\gamma$ le $\theta$-uplet défini par $\gamma=(\gamma_{1},\dots,\gamma_{\theta})=(\underbrace{\beta_{1},\dots,\beta_{1}}_{\alpha_{1} \text{ fois}},\dots,\underbrace{\beta_{n},\dots,\beta_{n}}_{\alpha_{n} \text{ fois}})$. 
	\begin{enumerate}
		\item\label{pointun} Si $n\geq3$ alors les éléments $\Gamma_{\gamma}$ et $\Gamma_{\Lambda_{\gamma}}$ ne sont pas des éléments primitifs de $\Cp$.
		\item\label{pointdeux} Si $n=2$ et $(\beta_{1},\beta_{2})\notin\left\{(1,2),(2,1)\right\}$ alors les éléments $\Gamma_{\gamma}$ et $\Gamma_{\Lambda_{\gamma}}$ ne sont pas des éléments primitifs de $\Cp$.
		\item\label{pointtrois} On suppose $n=2$. Si $\big((\beta_{1},\beta_{2})=(1,2)\text{ et }\alpha_{2}\geq2\big)$ ou $\big((\beta_{1},\beta_{2})=(2,1)\text{ et }\alpha_{1}\geq2\big)$ alors les éléments $\Gamma_{\gamma}$ et $\Gamma_{\Lambda_{\gamma}}$ ne sont pas des éléments primitifs de $\Cp$.
	\end{enumerate}
\end{prop}

\begin{proof} 
	Le point \ref{pointdeux} est une conséquence directe de la proposition \ref{conditionEcartMax}. Pour le point \ref{pointun}, s'il n'existe aucune permutation $\varsigma\in\Sn{n}$ telle que $(\beta_{1},\dots,\beta_{n})=(\varsigma(1),\dots,\varsigma(n))$, il s'agit alors d'une conséquence directe de la proposition \ref{conditionEcartMax}. Si, au contraire, il existe une permutation $\varsigma\in\Sn{n}$ telle que $(\beta_{1},\dots,\beta_{n})=(\varsigma(1),\dots,\varsigma(n))$, le point \ref{pointun} peut se montrer de la même façon que le point \ref{pointtrois} \emph{i.e.} en explicitant des tenseurs particuliers  de $\Cp\otimes\Cp$ dont la multiplicité est non nulle dans $\De(\Gamma_{\gamma})$ (respectivement $\De(\Gamma_{\Lambda_{\gamma}})$). 
\end{proof}

\Cex Les éléments suivants ne sont pas primitifs dans $\Cp$.
\begin{align*}
\Gamma_{(1,3,5)}&=(0,1,3,5)-2(0,3,1,5)+(0,3,5,1)+(0,1,5,3)-2(0,5,1,3)+(0,5,3,1),\\
\Gamma_{(1,1,2,3)}&=(0,1,1,2,3)-2(0,1,2,1,3)+2(0,2,1,3,1)-(0,2,3,1,1)\\
&+(0,1,1,3,2)-2(0,1,3,1,2)+2(0,3,1,2,1)-(0,3,2,1,1),\\
\Gamma_{\Lambda_{(1,1,1,2,2)}}&=3(0,1,1,1,2,2)-7(0,1,1,2,1,2)+3(0,1,2,1,1,2)-2(0,2,1,1,1,2)\\
&-2(0,1,1,2,2,1)+8(0,1,2,1,2,1)+3(0,2,1,1,2,1)-2(0,1,2,2,1,1)\\
&-7(0,2,1,2,1,1)+3(0,2,2,1,1,1).
\end{align*}

\begin{prop}\label{primCp}
	Soit $n\in\N^{*}$. Considérons l'élément	
	$$\Gamma_{(2,1,\dots,1)}=\sum_{k=0}^{n}(-1)^{k}\binom{n}{k}(0,\underbrace{1,\dots,1}_{k \text{ fois}},2,\underbrace{1,\dots,1}_{n-k \text{ fois}}).$$
	La famille $(\Gamma_{(2,1,\dots,1)\in(\N)^{n+1}})_{n\in\N^{*}}$ est une famille d'éléments primitifs de $\Cp$.
\end{prop}

\begin{proof}
	Soit $n\in\N^{*}$. Les seuls tenseurs pouvant potentiellement intervenir dans le coproduit réduit de $\Gamma_{(2,1,\dots,1)}$ sont ceux du type 
	\begin{align*}
		(0,\underbrace{1,\dots,1}_{s \text{ fois}})\ot (0,\underbrace{1,\dots,1}_{u \text{ fois}},2,\underbrace{1,\dots,1}_{v \text{ fois}})& \text{ avec } (s,u,v)\in\N^{*}\times(\N)^{2} \text{ et } s+u+v=n,\\
		(0,\underbrace{1,\dots,1}_{u \text{ fois}},2,\underbrace{1,\dots,1}_{v \text{ fois}})\ot (0,\underbrace{1,\dots,1}_{s \text{ fois}})& \text{ avec } (s,u,v)\in\N^{*}\times(\N)^{2} \text{ et } s+u+v=n,\\
		(0,\underbrace{1,\dots,1}_{s \text{ fois}})\ot (1,\underbrace{1,\dots,1}_{n+1-s \text{ fois}})& \text{ avec } 1\leq s\leq n+1.
	\end{align*}
	D'après les propositions \ref{lienCpISPW} (page \pageref{lienCpISPW}) et \ref{prim} (page \pageref{prim}), on sait que les coefficients des tenseurs  
	$(0,\underbrace{1,\dots,1}_{s \text{ fois}})\ot (0,\underbrace{1,\dots,1}_{u \text{ fois}},2,\underbrace{1,\dots,1}_{v \text{ fois}})  \text{ et } (0,\underbrace{1,\dots,1}_{u \text{ fois}},2,\underbrace{1,\dots,1}_{v \text{ fois}})\ot(0,\underbrace{1,\dots,1}_{s \text{ fois}})$ sont nuls. Il ne reste donc qu'à déterminer le coefficient $c_{s}$ du tenseur $t_{s}=(0,\underbrace{1,\dots,1}_{s \text{ fois}})\ot (1,\underbrace{1,\dots,1}_{n+1-s \text{ fois}})$ pour tout entier $s$ de $\llbracket1,n+1\rrbracket$. Soient $k\in\llbracket0,n\rrbracket$, $u\in\llbracket0,s\rrbracket$. On s'intéresse à l'élément $\omega_{k}=(0,\underbrace{1,\dots,1}_{k \text{ fois}},2,\underbrace{1,\dots,1}_{n-k \text{ fois}})$. Il est possible, à partir de cet élément, d'obtenir le tenseur $t_{s}$. Il suffit d'extraire, dans $\omega_{k}$, un nombre $u$ de positions à gauche de la $k+1$-ième et $s-1-u$ positions à droite. Cela ajoute des conditions supplémentaires sur $u$, à savoir $s+k-n-1\leq u \leq k$. On peut maintenant calculer $c_{s}$. 
	On a: $c_{s}=\displaystyle\sum_{k=0}^{n}(-1)^{k}\binom{n}{k}\sum_{u=\max(0,s+k-n-1)}^{\min(s-1,k)}2\binom{k}{u}\binom{n-k}{s-1-u}=0$. 
	L'élément $\Gamma_{(2,1,\dots,1)}$ est donc primitif dans $\Cp$.
\end{proof}
\Ex{}
\begin{align*}
	\Gamma_{(2,1,1)}&=(0,2,1,1)-2(0,1,2,1)+(0,1,1,2),\\
	\Gamma_{(2,1,1,1)}&=(0,2,1,1,1)-3(0,1,2,1,1)+3(0,1,1,2,1)-(0,1,1,1,2).
\end{align*}
\begin{cor}Les propositions \ref{lienCpISPW} à \ref{primCp} permettent d'établir que: \begin{align*}
		\Prim{\Cp}{4}&=Vect(\Gamma_{(1,1,2)}),& \Prim{\Cp}{5}&=Vect(\Gamma_{(1,1,1,2)}),\\ \Prim{\Cp}{6}&=Vect(\Gamma_{(1,1,1,1,2)}).&&
	\end{align*}
\end{cor}

\Rq{1} On connaît également l'espace vectoriel $\Prim{\Cp}{7}$. En effet, l'élément $u$ défini par $u=(0,3,2,1,1)-2(0,3,1,2,1)+(0,3,1,1,2)-(0,2,1,1,3)+2(0,1,2,1,3)-(0,1,1,2,3)$ n'est pas un élément primitif de $\Cp$. On obtient alors que $\Prim{\Cp}{7}=Vect(\Gamma_{(2,1,1,1,1,1)})$ grâce aux propositions précédentes.

\subsubsection{Algèbre de Hopf duale $\dual{\Cp}$}
On s'intéresse à $\dual{\Cp}$ le dual de Hopf gradué de $\Cp$. Grâce à l'isomorphisme $\Pi$ entre $\WMat/\Idl{J}$ et $\Cp$, on peut également voir  $\dual{\Cp}$ comme  $\dual{(\WMat/\Idl{J})}$ donc comme une sous-algèbre de Hopf de $\dual{\WMat}$.  
On note  $(Z_{\alpha})_{\alpha\in\Cp }$ la base duale des compositions étendues.

Décrivons les opérations de l'algèbre de Hopf $\dual{\Cp}$. Commençons par la structure la plus simple: celle de cogèbre.
\begin{prop}
	Le coproduit de la cogèbre $\dual{\Cp}$ est défini par: 
	$$\De:\left\{\begin{array}{rcl} \dual{\Cp}&\longrightarrow & \dual{\Cp}\ot\dual{\Cp} \\ Z_{(\alpha_{0},\dots,\alpha_{s})} &\longrightarrow &\sum\limits_{a=0}^{\alpha_{0}}\sum\limits_{u=0}^{s}Z_{(a,\alpha_{1},\dots,\alpha_{u})}\ot Z_{(\alpha_{0}-a,\alpha_{u+1}\dots,\alpha_{s})}\end{array}\right.$$	
\end{prop}
\begin{proof}
	Soient $\alpha=(\alpha_{0},\alpha_{1},\dots,\alpha_{s}), \nu=(\nu_{0},\nu_{1},\dots,\nu_{p}), \eta=(\eta_{0},\eta_{1},\dots,\eta_{k}) \in \Cp$ des compositions étendues. On a :
	\begin{align*}
	\Delta(Z_{\alpha})(\nu\otimes\eta)=& Z_{\alpha}(\nu_{0}+\eta_{0},\nu_{1},\dots \nu_{p},\eta_{1}\dots \eta_{k})\\
	=&\delta_{\alpha_{0},\nu_{0}+\eta_{0}}\delta_{\alpha_{1},\nu_{1}}\dots\delta_{\alpha_{p},\nu_{p}}\delta_{\alpha_{p+1},\eta_{1}}\dots\delta_{\alpha_{s},\eta_{k}}\\
	=&\sum\limits_{a=0}^{\alpha_{0}}\sum\limits_{u=0}^{s}\left(Z_{(a,\alpha_{1},\dots,\alpha_{u})}\ot Z_{(\alpha_{0}-a,\alpha_{u+1}\dots,\alpha_{s})}\right)(\nu\ot \eta)
	\end{align*}
	Donc $$\Delta(Z_{(\alpha_{0},\dots,\alpha_{s})})=\sum_{a=0}^{\alpha_{0}}\sum_{u=0}^{s}Z_{(a,\alpha_{1},\dots,\alpha_{u})}\ot Z_{(\alpha_{0}-a,\alpha_{u+1}\dots,\alpha_{s})}.$$
\end{proof}
\Ex{} Considérons les compositions étendues $(0,2,1,1,2)$ et $(3,2,2,4)$. On détermine leur coproduit réduit. On a alors:
\begin{align*}
	\tilde{\De}(Z_{(0,2,1,1,2)})=&Z_{(0,2)}\ot Z_{(0,1,1,2)}+ Z_{(0,2,1)}\ot Z_{(0,1,2)}+ Z_{(0,2,1,1)}\ot Z_{(0,2)},\\
	\tilde{\De}(Z_{(1,2,3)})=&Z_{(0,2)}\ot Z_{(1,3)} +Z_{(0,2,3)}\ot Z_{(1)} + Z_{(1)}\ot Z_{(0,2,3)} + Z_{(1,2)}\ot Z_{(0,3)}.
\end{align*}
	
Intéressons-nous maintenant à la structure d'algèbre de $\dual{\Cp}$.
\begin{prop}
	Le produit de l'algèbre $\dual{\Cp}$ est défini par: 
	$$m:\left\{\begin{array}{rcl} \dual{\Cp}\ot\dual{\Cp} &\longrightarrow & \dual{\Cp}\\ Z_{(\alpha_{0},\dots,\alpha_{p})}\ot Z_{(\beta_{0},\dots,\beta_{q})} &\longrightarrow &\sum\limits_{\mu=0}^{\beta_{0}}\sum\limits_{\gamma=(\gamma_{1},\dots,\gamma_{p})\in K_{p,\beta_{0}-\mu}}c_{\mu,\gamma}Z_{(\alpha_{0}+\mu,(\alpha_{1}+\gamma_{1},\dots,\alpha_{p}+\gamma_{p})\shuffle (\beta_{1},\dots,\beta_{q}) )}\end{array}\right.$$	
	où
	\begin{align*}
	K_{p,m}=&\bigg\{(\gamma_{1},\dots,\gamma_{p})\in\N^{p}, \gamma_{1}+\dots+\gamma_{p}=m  \bigg \} \text{ pour } m\in\N,\\
	c_{\mu,\gamma}=&\binom{\alpha_{0}+\mu}{\alpha_{0}}\binom{\alpha_{1}+\gamma_{1}}{\alpha_{1}}\dots\binom{\alpha_{p}+\gamma_{p}}{\alpha_{p}},\\
	Z_{(\alpha_{0}+\mu,(\alpha_{1}+\gamma_{1},\dots,\alpha_{p}+\gamma_{p})\shuffle (\beta_{1},\dots,\beta_{q}) )}=& \sum\limits_{\tau\in Bat(p,q)}Z_{(\alpha_{0}+\mu,\varphi_{\tau^{-1}(1)},\dots,\varphi_{\tau^{-1}(p)},\varphi_{\tau^{-1}(p+1)},\dots,\varphi_{\tau^{-1}(p+q)})},
	\end{align*}
	avec
	$$\varphi_{i}=\begin{cases}
	\alpha_{i}+\gamma_{i} & \mbox{ si } i\in\llbracket1,p\rrbracket, \\
	\beta_{i-p} & \mbox{ si } i\in\llbracket p+1,p+q\rrbracket. 
	\end{cases}$$
\end{prop}
\begin{proof}
	Soient $\alpha=(\alpha_{0},\alpha_{1},\dots,\alpha_{p}), \beta=(\beta_{0},\beta_{1},\dots,\beta_{q}), \nu=(\nu_{0},\nu_{1},\dots,\nu_{s})\in \Cp$  des compositions étendues.
	On a : 
	\begin{align*} 
	& (Z_{\alpha}Z_{\beta})(\nu)=\\
	& \sum\limits_{a=0}^{\nu_{0}}\sum\limits_{n=0}^{s}\sum\limits_{1\leq i_{1}<\dots<i_{n}\leq s}\sum\limits_{t=n}^{\nu_{i_{1}}+\dots+\nu_{i_{n}}}\sum\limits_{\substack{k_{i_{1}}+\dots+k_{i_{n}}=t \\ 1\leq k_{i_{j}}\leq \nu_{i_{j}}}}\binom{\nu_{0}}{a}\binom{\nu_{i_{1}}}{k_{i_{1}}}\dots\binom{\nu_{i_{n}}}{k_{i_{n}}}Z_{\alpha}(v_{a,I,K})Z_{\beta}(w_{a,I,K})
	\end{align*}
	où \begin{align*}
	v_{a,I,K}=&(a,k_{i_{1}},\dots,k_{i_{n}}),\\
	w_{a,I,K}=&(\nu_{0}-a+\sum\limits_{j=1}^{n}(\nu_{i_{j}}-k_{i_{j}}),\nu_{u_{1}},\dots,\nu_{u_{s-n}}),\\
	\{u_{1}<\dots<u_{s-n}\}=&\llbracket1,s\rrbracket\setminus\{i_{1}\dots,i_{n}\}.
	\end{align*}
	Or $$Z_{\alpha}(v_{a,k,i})Z_{\beta}(w_{a,k,i})\neq 0 \iff \left\{ \begin{array}{rcl}
	a &=& \alpha_{0}\\
	s &=& p+q\\
	n &=& p \\
	k_{i_{j}} &=& \alpha_{j}, \forall j\in\llbracket1,p\rrbracket\\
	\nu_{u_{j}} &=& \beta_{j}, \forall j\in\llbracket1,q\rrbracket\\
	\nu_{0}+\sum\limits_{j=1}^{p}\nu_{i_{j}} &=& \beta_{0}+\alpha_{0}+\sum\limits_{j=1}^{p}\alpha_{j}\\
	0<\alpha_{j} &\leq& \nu_{i_{j}}, \forall j\in\llbracket1,p\rrbracket
	\end{array} \right.$$
	On obtient donc le résultat énoncé.
\end{proof}
\Ex{}
\begin{align*}
	Z_{(0,1)}Z_{(0,1)}=&2Z_{(0,1,1)},\\
	Z_{(0,1)}Z_{(1,1)}=&2Z_{(0,2,1)}+2Z_{(0,1,2)}+2Z_{(1,1,1)}.
\end{align*}

\section{Isomorphisme entre $\dual{\ISPW}$ et $\QSym$, isomorphisme entre $\ISPW$ et $\NSym$}\label{morphQSym}
Dans cette section, l'objectif est d'expliciter un isomorphisme d'algèbres de Hopf entre $\dual{\ISPW}$ et $\QSym$ (\emph{cf.} paragraphe \ref{QSymdef}), et donc un isomorphisme entre $\NSym$ et $\ISPW$ selon un procédé décrit par Aguiar, Bergeron et Sottile dans \cite[théorème 4.1]{Aguiar2006}.  

\subsection{Rappels.} Effectuons quelques rappels sur le procédé de construction de morphismes introduit par Aguiar, Bergeron et Sottile dans \cite{Aguiar2006}. 
Dans l'algèbre de Hopf $\QSym$, on utilise la base définie dans le paragraphe \ref{QSymdef} à savoir $(M_{\alpha})_{\alpha=(\alpha_{1},\dots,\alpha_{k})\in\N^{k}}$. 
Ici, on appellera algèbre de Hopf combinatoire une algèbre de Hopf $H$ munie d'un caractère $\zeta$. Considèrons donc une algèbre de Hopf combinatoire  $(H,\zeta)$ ainsi que l'algèbre de Hopf combinatoire $(\QSym,\zeta_{Q})$ où $\zeta_{Q}$ est le caractère de $\QSym$ défini par: 
$$\forall \alpha=(\alpha_{1},\dots,\alpha_{k}) \mbox{, } \zeta_{Q}(M_{\alpha})=\begin{cases}
1 & \mbox{ si } k=0 \mbox{ ou } k=1,\\
0 & \mbox{ sinon. } 
\end{cases}$$
On sait (\emph{ cf.} \cite[théorème 4.1]{Aguiar2006}) qu'il existe alors un unique morphisme d'algèbres de Hopf graduées
$$\Psi :(H,\zeta) \longrightarrow (\QSym,\zeta_{Q})$$
tel que $\zeta_{Q}\circ\Psi=\zeta$. De plus, pour tout élément $h\in(H)_{n}$ de degré $n$, $\Psi(h)=\sum\limits_{\alpha\models n}\zeta_{\alpha}(w)M_{\alpha}$ où, pour $\alpha=(\alpha_{1},\dots,\alpha_{k})$, $\zeta_{\alpha}$ est la composée
$$H \xrightarrow{\Delta^{(k-1)}} H^{\otimes k}\twoheadrightarrow(H)_{\alpha_{1}}\otimes\dots\otimes(H)_{\alpha_{k}}\xrightarrow{\zeta^{\otimes k}}\K.$$

Considérons maintenant l'algèbre de Hopf $\dual{\ISPW}$. Définissons un caractère $\zeta$ permettant d'obtenir, grâce au procédé ci-dessus, un isomorphisme d'algèbres de Hopf entre $\dual{\ISPW}$ et $\QSym$. Par transposition de cette application nous aurons ainsi un isomorphisme d'algèbres de Hopf explicite entre $\NSym$ et $\ISPW$.

\subsection{Application à l'algèbre de Hopf $\dual{\ISPW}$. Isomorphisme entre $\ISPW$ et $\NSym$}
Le procédé explicité par Aguiar, Bergeron et Sottile permet de définir un isomorphisme d'algèbres de Hopf $\Psi$ entre $\dual{\ISPW}$ et $\QSym$. Pour cela, on utilisera l'écriture en composition des éléments de $\dual{\ISPW}$ \emph{i.e.}, pour tout entier $n\in\N^{*}$ et tout $n$-uplet $(k_{1},\dots,k_{n})\in(\N^{*})^{n}$, l'élément $Z_{\underbrace{x_{1}\dots x_{1}}_{k_{1} \text{ fois}}\dots\underbrace{x_{n}\dots x_{n}}_{k_{n}}}$ est noté $Z_{(k_{1},\dots,k_{n})}$.
\begin{prop}
	L'application 
	$$\zeta:\left\{\begin{array}{rcl} \dual{\ISPW} &\longrightarrow & \K \\Z_{(k_{1},\dots,k_{n})} &\longrightarrow & \cfrac{1}{n!}\end{array}\right.$$
	est un caractère de $\dual{\ISPW}$.
\end{prop}

\begin{proof}
	Le produit de deux compositions $Z_{(k_{1},\dots,k_{n})}$ et $Z_{(l_{1},\dots,l_{m})}$ de $\dual{\ISPW}$ est une somme de $\displaystyle \binom{n+m}{n}$ compositions de longueur $n+m$. Ainsi, $$\zeta(Z_{(k_{1},\dots,k_{n})}Z_{(l_{1},\dots,l_{m})})=\frac{\binom{n+m}{n}}{(n+m)!}=\frac{1}{n!m!}=\zeta(Z_{(k_{1},\dots,k_{n})})\zeta(Z_{(l_{1},\dots,l_{m})}).$$
\end{proof}

\begin{prop}\label{isoISPWdQSym}
	L'application $\Psi$, définie par 
	$$\Psi:\left\{\begin{array}{rcl} \dual{\ISPW} &\longrightarrow & \QSym \\Z_{(k_{1},\dots,k_{n})} &\longrightarrow & \sum\limits_{i=0}^{n-1}\sum\limits_{(s_{1},\dots,s_{i+1})\models n}\cfrac{1}{s_{1}!\dots s_{i+1}!}M_{(k_{1}+\dots+k_{s_{1}}~,~\dots~,~k_{(s_{1}+\dots+s_{i}+1)}+\dots +k_{n})},\end{array}\right.$$
	est un isomorphisme d'algèbres de Hopf.
\end{prop}

\begin{proof}
	Soient $Z_{(k_{1},\dots,k_{n})}$ un élément de $\dual{\ISPW}$ et $i$ un entier naturel de l'intervalle $\llbracket 1, n-1 \rrbracket$. Notons $\tilde{\De}^{i}$, le coproduit réduit $\tilde{\De}$ itéré $i$ fois. On sait que:
	$$\tilde{\De}^{i}(Z_{(k_{1},\dots,k_{n})})=\sum_{(s_{1},\dots,s_{i+1})\models n}Z_{(k_{1},\dots,k_{s_{1}})}\ot\dots\ot Z_{(k_{(s_{1}+\dots+s_{i}+1)},\dots,k_{n})}.$$
	On considère alors le caractère $\zeta$ défini dans la proposition précédente et le procédé de Aguiar, Bergeron et Sottile donne l'application $\Psi$. Cette dernière est donc un morphisme d'algèbres de Hopf graduées. En considérant la base $(Z_{(k_{1},\dots,k_{n})})_{(n,k_{1},\dots,k_{n})\in(\N^{*})^{n+1}}$, le fait que $\Psi$ soit un isomorphisme est immédiat.  
\end{proof}

\begin{prop}\label{isoNSymISPW}
	L'application $\Psi^{*}$, définie par
	$$\Psi^{*}:\left\{\begin{array}{rcl} \NSym &\longrightarrow & \ISPW \\M^{*}_{(k_{1},\dots,k_{n})} &\longrightarrow & \sum\limits_{i=n}^{k_{1}+\dots+k_{n}}\sum\limits_{\substack{s_{1}+\dots+s_{n}=i \\ \forall j,~1\leq s_{j}\leq k_{j}}}\sum\limits_{\substack{(u_{1}^{(1)},\dots,u_{s_{1}}^{(1)})\models k_{1}\\ \vdots \\ (u_{1}^{(n)},\dots,u_{s_{n}}^{(n)})\models k_{n}}}\cfrac{1}{s_{1}!\dots s_{n}!}w_{U},\end{array}\right.$$
	avec $w_{U}=\underbrace{x_{1}\dots x_{1}}_{u_{1}^{(1)} \text{ fois}}\ast\dots\ast \underbrace{x_{1}\dots x_{1}}_{u_{s_{1}}^{(1)} \text{ fois}}\ast\dots\ast\underbrace{x_{1}\dots x_{1}}_{u_{1}^{(n)} \text{ fois}}\ast\dots\ast \underbrace{x_{1}\dots x_{1}}_{u_{s_{n}}^{(n)} \text{ fois}}$,
	est un isomorphisme d'algèbres de Hopf.	
\end{prop}
\begin{proof}
	Il suffit de transposer l'application $\Psi$ précédente.
\end{proof}

%\textbf{Remerciements.} Je remercie le rapporteur anonyme de cette publication pour toutes ses remarques et suggestions. Je remercie également Yannic Vargas de m'avoir précisé les origines des algèbres de permutations de Aguiar et Mahajan définies grâce aux foncteurs de Fock.

\bibliographystyle{siam}
\bibliography{biblio_mots_tasses}

\begin{thebibliography}{10}

\bibitem{Aguiar2006}
{\sc M.~Aguiar, N.~Bergeron, and F.~Sottile}, {\em {C}ombinatorial {H}opf
  algebras and generalized {D}ehn-{S}ommerville relations}, Compos. Math., 142
  (2006), pp.~1--30.

\bibitem{Aguiar2004}
{\sc M.~Aguiar and J.-L. Loday}, {\em {Q}uadri-algebras}, J. Pure Appl.
  Algebra, 191 (2004), pp.~205--221.

\bibitem{Aguiar2010}
{\sc M.~Aguiar and S.~Mahajan}, {\em Monoidal functors, species and {H}opf
  algebras}, vol.~29 of CRM Monograph Series, American Mathematical Society,
  Providence, RI, 2010.
\newblock With forewords by Kenneth Brown and Stephen Chase and Andr\'e Joyal.

\bibitem{Aguiar2006a}
{\sc M.~Aguiar and F.~Sottile}, {\em {S}tructure of the {L}oday-{R}onco {H}opf
  algebra of trees}, J. Algebra, 295 (2006), pp.~473--511.

\bibitem{Calaque2011}
{\sc D.~Calaque, K.~Ebrahimi-Fard, and D.~Manchon}, {\em {T}wo interacting
  {H}opf algebras of trees: a {H}opf-algebraic approach to composition and
  substitution of {B}-series}, Adv. in Appl. Math., 47 (2011), pp.~282--308.

\bibitem{Chapoton2002}
{\sc F.~Chapoton}, {\em {U}n théorème de {C}artier-{M}ilnor-{M}oore-{Q}uillen
  pour les bigèbres dendriformes et les algèbres braces}, J. Pure Appl.
  Algebra, 168 (2002), pp.~1--18.

\bibitem{Chapoton2005}
\leavevmode\vrule height 2pt depth -1.6pt width 23pt, {\em {O}n some anticyclic
  operads}, Algebr. Geom. Topol., 5 (2005), pp.~53--69.

\bibitem{Connes1999}
{\sc A.~Connes and D.~Kreimer}, {\em {H}opf algebras, renormalization and
  noncommutative geometry}, in Quantum field theory: perspective and
  prospective (Les Houches, 1998), vol.~530 of NATO Sci. Ser. C Math. Phys.
  Sci., Kluwer Acad. Publ., Dordrecht, 1999, pp.~59--108.

\bibitem{Connes2000}
\leavevmode\vrule height 2pt depth -1.6pt width 23pt, {\em {R}enormalization in
  quantum field theory and the {R}iemann-{H}ilbert problem {I}: {T}he {H}opf
  algebra structure of graphs and the main theorem}, Comm. Math. Phys., 210
  (2000), pp.~249--273.

\bibitem{Duchamp2011}
{\sc G.~Duchamp, F.~Hivert, J.-C. Novelli, and J.-Y. Thibon}, {\em
  {N}oncommutative symmetric functions {VII}: {F}ree quasi-symmetric functions
  revisited}, Ann. Comb., 15 (2011), pp.~655--673.

\bibitem{Duchamp2002}
{\sc G.~Duchamp, F.~Hivert, and J.-Y. Thibon}, {\em {N}oncommutative symmetric
  functions {VI}: {F}ree quasi-symmetric functions and related algebras},
  Internat. J. Algebra Comput., 12 (2002), pp.~671--717.

\bibitem{Duchamp2013}
{\sc G.~Duchamp, N.~Hoang-Nghia, and A.~Tanasa}, {\em {A} selection-quotient
  process for packed word {H}opf algebra}, in Algebraic informatics, vol.~8080
  of Lecture Notes in Comput. Sci., Springer, Heidelberg, 2013, pp.~223--234.

\bibitem{Duchamp1997}
{\sc G.~Duchamp, A.~Klyachko, D.~Krob, and J.-Y. Thibon}, {\em {N}oncommutative
  symmetric functions {III}: {D}eformations of cauchy and convolution
  algebras}, Discrete Math. Theor. Comput. Sci., 1 (1997), pp.~159--216.
\newblock Lie computations (Marseille, 1994).

\bibitem{Dupont2014}
{\sc C.~Dupont}, {\em {P}ériodes des arrangements d'hyperplans et coproduit
  motivique}, PhD thesis, 2014.

\bibitem{Ehrenborg1996}
{\sc R.~Ehrenborg}, {\em {O}n posets and {H}opf algebras}, Adv. Math., 119
  (1996), pp.~1--25.

\bibitem{Foissy2002b}
{\sc L.~Foissy}, {\em {L}es algèbres de {H}opf des arbres enracinés décorés.
  {II}}, Bull. Sci. Math., 126 (2002), pp.~249--288.

\bibitem{Foissy2007}
\leavevmode\vrule height 2pt depth -1.6pt width 23pt, {\em {B}idendriform
  bialgebras, trees and free quasi-symmetric functions}, J. Pure Appl. Algebra,
  209 (2007), pp.~439--459.

\bibitem{Foissy2015}
\leavevmode\vrule height 2pt depth -1.6pt width 23pt, {\em {F}ree
  quadri-algebras and dual quadri-algebras}, ArXiv e-prints,  (2015).

\bibitem{Garsia1989}
{\sc A.~M. Garsia and C.~Reutenauer}, {\em {A} decomposition of {S}olomon's
  descent algebra}, Adv. Math., 77 (1989), pp.~189--262.

\bibitem{Gelfand1995}
{\sc I.~M. Gelfand, D.~Krob, A.~Lascoux, B.~Leclerc, V.~S. Retakh, and J.-Y.
  Thibon}, {\em {N}oncommutative symmetric functions}, Adv. Math., 112 (1995),
  pp.~218--348.

\bibitem{Gessel1984}
{\sc I.~M. Gessel}, {\em {M}ultipartite ${P}$-partitions and inner products of
  skew {S}chur functions}, in Combinatorics and algebra (Boulder, Colo., 1983),
  vol.~34 of Contemp. Math., Amer. Math. Soc., Providence, RI, 1984,
  pp.~289--317.

\bibitem{Krob1997}
{\sc D.~Krob, B.~Leclerc, and J.-Y. Thibon}, {\em {N}oncommutative symmetric
  functions {II}: {T}ransformations of alphabets}, Internat. J. Algebra
  Comput., 7 (1997), pp.~181--264.

\bibitem{Krob1997a}
{\sc D.~Krob and J.-Y. Thibon}, {\em {N}oncommutative symmetric functions {IV}:
  {Q}uantum linear groups and hecke algebras at q=0}, J. Algebraic Combin., 6
  (1997), pp.~339--376.

\bibitem{Krob1999}
\leavevmode\vrule height 2pt depth -1.6pt width 23pt, {\em {N}oncommutative
  symmetric functions {V}: {A} degenerate version of {$U_q({\rm gl}_N)$}},
  Internat. J. Algebra Comput., 9 (1999), pp.~405--430.
\newblock Dedicated to the memory of Marcel-Paul Sch\"utzenberger.

\bibitem{Livernet1998}
{\sc M.~Livernet}, {\em {R}ational homotopy of {L}eibniz algebras}, manuscripta
  mathematica, 96 (1998), pp.~295--315.

\bibitem{Loday1995}
{\sc J.-L. Loday}, {\em {C}up-product for {L}eibniz cohomology and dual
  {L}eibniz algebras}, Math. Scand., 77 (1995), pp.~189--196.

\bibitem{Loday2001}
\leavevmode\vrule height 2pt depth -1.6pt width 23pt, {\em {D}ialgebras}, in
  Dialgebras and related operads, vol.~1763 of Lecture Notes in Math.,
  Springer, Berlin, 2001, pp.~7--66.

\bibitem{Loday1998}
{\sc J.-L. Loday and M.~Ronco}, {\em {H}opf algebra of the planar binary
  trees}, Adv. Math., 139 (1998), pp.~293--309.

\bibitem{Malvenuto1994}
{\sc C.~Malvenuto}, {\em {P}roduits et coproduits des fonctions
  quasi-symétriques et de l'algèbre des descentes}, PhD thesis, 1994.

\bibitem{Malvenuto1995}
{\sc C.~Malvenuto and C.~Reutenauer}, {\em {D}uality between quasi-symmetric
  functions and the {S}olomon descent algebra}, J. Algebra, 177 (1995),
  pp.~967--982.

\bibitem{Mammezb}
{\sc C.~Mammez}, {\em {D}eux exemples d'algèbres de Hopf
  d'extraction-contraction: mots tassés et diagrammes de dissection}, PhD
  thesis.

\bibitem{Manchon2012}
{\sc D.~Manchon}, {\em {O}n bialgebras and {H}opf algebras or oriented graphs},
  Confluentes Math., 4 (2012), pp.~1240003, 10.

\bibitem{Markopoulou2003}
{\sc F.~Markopoulou}, {\em {C}oarse graining in spin foam models}, Classical
  Quantum Gravity, 20 (2003), pp.~777--799.

\bibitem{Molnar1977}
{\sc R.~K. Molnar}, {\em {S}emi-direct products of {H}opf algebras}, J.
  Algebra, 47 (1977), pp.~29--51.

\bibitem{Ronco2001}
{\sc M.~Ronco}, {\em {A} {M}ilnor-{M}oore theorem for dendriform hopf
  algebras}, C. R. Acad. Sci. Paris S\'er. I Math., 332 (2001), pp.~109--114.

\bibitem{Sloane}
{\sc N.~J.~A. Sloane}, {\em The on-line encyclopedia of integer sequences},
  1964.

\bibitem{Solomon1976}
{\sc L.~Solomon}, {\em {A} {M}ackey formula in the group ring of a {C}oxeter
  group}, J. Algebra, 41 (1976), pp.~255--264.

\bibitem{Stanley1972}
{\sc R.~P. Stanley}, {\em {O}rdered structures and partitions}, American
  Mathematical Society, Providence, R.I., 1972.
\newblock Memoirs of the American Mathematical Society, No. 119.

\bibitem{Tanasa2013}
{\sc A.~Tanasa and D.~Kreimer}, {\em {C}ombinatorial {D}yson-{S}chwinger
  equations in noncommutative field theory}, J. Noncommut. Geom., 7 (2013),
  pp.~255--289.

\bibitem{Vallette2008}
{\sc B.~Vallette}, {\em {M}anin products, {K}oszul duality, {L}oday algebras
  and {D}eligne conjecture}, J. Reine Angew. Math., 620 (2008), pp.~105--164.

\bibitem{Vargas2014}
{\sc Y.~Vargas}, {\em {H}opf algebra of permutation pattern functions}, in
  {26th {I}nternational {C}onference on {F}ormal {P}ower {S}eries and
  {A}lgebraic {C}ombinatorics ({FPSAC} 2014)}, L.~J. Billera and I.~Novik,
  eds., vol.~AT of DMTCS Proceedings, Chicago, United States, 2014, Discrete
  Mathematics and Theoretical Computer Science, pp.~839--850.

\end{thebibliography}
\end{document}